\documentclass{amsart}
\usepackage{tikz}
\usepackage{color}
\usepackage{amsfonts}
\usepackage{amsmath,amsthm,amssymb,esint}
\usepackage[left=3.5cm,right=3.5cm,top=3cm,bottom=3cm]{geometry}
\newtheorem{theorem}{Theorem}[section]
\newtheorem{lemma}[theorem]{Lemma}
\newtheorem{proposition}[theorem]{Proposition}
\newtheorem{corollary}[theorem]{Corollary}

\theoremstyle{definition}
\newtheorem{definition}[theorem]{Definition}

\theoremstyle{remark}
\newtheorem{remark}[theorem]{Remark}

\numberwithin{equation}{section}

\newcommand {\vv}[1]{{\mathsf #1}}
\newcommand{\alg}[1]{{\textbf{\upshape #1}}}

\newcommand{\pp}{\mathbf{P}}

\newcommand{\II}{\mathbf{I}}

\newcommand{\back}{\backslash}

\begin{document}

\title{Gluing residuated lattices}


\author{Nikolaos Galatos}
\address{Department of Mathematics, University of Denver, CO, USA}
\curraddr{}
\email{nikolaos.galatos@du.edu}
\thanks{}

\author{Sara Ugolini}
\address{IIIA - CSIC, Bellaterra, Barcelona, Spain}
\curraddr{}
\email{sara@iiia.csic.es}
\thanks{}



\date{}

\dedicatory{}

\begin{abstract}
We introduce and characterize various gluing constructions for residuated lattices that intersect on a common subreduct, and which are subalgebras, or appropriate subreducts, of the resulting structure. Starting from the $1$-sum construction (also known as ordinal sum for residuated structures), where algebras that intersect only in the top element are glued together, we first consider the gluing on a congruence filter, and then add a lattice ideal as well. We characterize such constructions in terms of (possibly partial) operators acting on (possibly partial) residuated structures. As particular examples of gluing constructions, we obtain the non-commutative version of some rotation constructions, and an interesting variety of semilinear residuated lattices that are $2$-potent.

This study also serves as a first attempt toward the study of amalgamation of non-commutative residuated lattices, by constructing an amalgam in the special case where the common subalgebra in the V-formation is either a special (congruence) filter or the union of a filter and an ideal.
\end{abstract}

\maketitle
\section{Introduction and preliminaries}
The first gluing construction in lattice theory is due to Hall and Dilworth \cite{HD44}, who used it to prove the existence of a modular lattice that cannot be embedded in a complemented modular lattice. Later on, the same construction was independently used by Wro\'nski in \cite{W74} and Troelstra \cite{T65} to study intermediate logics, by constructing Heyting algebras. 

The idea in  these constructions is to glue together lattices that intersect (up to isomorphism) on a sublattice that is a principal ideal of the first and a principal filter of the second.
In particular, the construction applies to Heyting algebras: bounded lattices that are relatively pseudocomplemented (i.e., for every pair of elements $x, y$ there is a largest element $z$ with the property that $x \land z \leq y$).
Heyting algebras can also be equivalently defined as bounded residuated lattices where the monoidal operation coincides with the meet in the lattice order; in this case $x \to y$ is the largest element $z$ such that $x \land z \leq y$.
Residuated lattices play an important role in the study of algebraic logic, as they constitute the equivalent algebraic semantics (in the sense of Blok-Pigozzi \cite{BP89}) of substructural logics. These encompass most interesting nonclassical logics: intuitionistic logic, fuzzy logics, relevance logics, linear logic, and classical logic as a limit case. Thus, the investigation of the variety of residuated lattices is a powerful tool in the comparative study of such logics, as explored in \cite{GJKO}. The multitude of different types of residuated lattices makes the study fairly complicated and at the present moment large classes of residuated lattices lack a structural description. The study of constructions that allow us to obtain new structures from known ones is extremely important in improving our understanding of residuated lattices, and as a result, of substructural logics. 

In the present paper we introduce different ways of gluing together residuated lattices, where by {\em gluing} we mean obtaining a new structure from two original ones which intersect on a common subreduct, and which are subalgebras (or appropriate subreducts) of the resulting structure. The starting point of our investigation is the {\em $1$-sum} construction, often called {\em ordinal sum} in residuated structures, where the algebras intersect only at the top element, i.e. at a trivial filter. We will consider gluings over an arbitrary  (nontrivial) congruence filter, and then over a lattice ideal as well. Moreover, we generalize these ideas to account for (possibly) partial algebras. Finally, we characterize the introduced constructions abstractly, by means of pairs of operators acting on residuated lattices.

These new constructions serve as a first attempt in the study of amalgamation of non-commutative residuated lattices, by constructing an amalgam in the special case where the common subalgebra in the V-formation is either a special (congruence) filter or the union of a filter and an ideal.

As particular examples of gluing constructions, we obtain the non-commutative version of the generalized rotation construction in \cite{BMU18}, and an interesting variety of semilinear residuated lattices that are $2$-potent. We use these two cases to illustrate examples of how the gluing construction can be used to study amalgamation. In the case of the rotation we show how the construction preserves the amalgamation property (in a sense that will be made precise), while in the latter case we show and characterize amalgamation failures.

\medskip
We start by introducing the objects of our study. 
A residuated lattices is an algebra $\alg A = ( A,\lor,\land,\cdot,\back, /, 1)$ of type $(2, 2, 2, 2, 2, 0)$ such that:
\begin{enumerate}
\item $(A, \lor, \land )$ is a lattice;
\item $(A, \cdot, 1)$ is a monoid;
\item $\back$ and $/$ are the left and right division of $\cdot$: for all $x,y, z \in A$,
$$
x \cdot y \leq z \Leftrightarrow y \leq x \back z \Leftrightarrow x \leq z /y,
$$
where $\leq$ is  the lattice ordering.
\end{enumerate}
Residuated lattices form a variety, denoted by $\mathsf{RL}$, as residuation can be expressed equationally; see \cite{BlountTsinakis}. When the monoidal identity is the top element of the lattice we say that the residuated lattice is \emph{integral} or an \emph{IRL}; we call the corresponding variety $\mathsf{IRL}$.
Residuated lattices with an additional constant $0$ are called \emph{pointed}. \emph{Bounded} integral residuated lattices are expansions of residuated lattice with an extra constant  $0$ that satisfies the identity $0 \leq x$.  The variety of bounded integral residuated lattices is called $\mathsf{FL_w}$, referring to the fact that it is the equivalent algebraic semantics of the Full Lambek calculus with the structural rule of weakening (see \cite{GJKO}). As usual, we write $xy$ for $x \cdot y$.

A residuated lattice is called \emph{commutative} if the monoidal operation is commutative. In this case the two divisions coincide, and we write $x \to y$ for $x \back y = y / x$. We write $\mathsf{CRL}$ and $\mathsf{CIRL}$, respectively, for the commutative subvarieties of $\mathsf{RL}$ and $\mathsf{IRL}$, and refer to commutative $\mathsf{FL_{w}}$-algebras as $\mathsf{FL_{ew}}$-algebras, since commutativity of the monoidal operation corresponds to the structural rule of exchange.

In a lattice $\alg A$, a \emph{filter} is a non-empty subset $S$ that is closed upwards (if $x \leq y$ and $x \in S$ then also $y \in S$) and is closed under meet (if $x, y \in S$ then $x \land y \in S$).
In a (bounded) integral residuated lattice $\alg A$, a \emph{congruence filter} $F$ is a non-empty upset of $A$, closed under products (if $x, y \in F$, then $x y \in F$) and under conjugates, i.e. if $x \in F$ , then $yx/y, y\backslash xy \in F$ for every $y \in A$. We will denote by $\alg {Fil}(\alg A)$ the lattice of congruence filters of $\alg A$. It is easy to see that a filter $F$ of a (bounded) integral residuated lattice $\alg A$ is a subalgebra (or $0$-free subreduct, if $\alg A$ is bounded) of $\alg A$, hence it is an integral residuated lattice.

Filters of residuated lattices are in one-one correspondence to congruences. In particular, in the integral case the isomorphism between $\alg {Fil}(\alg A)$ and the congruence lattice of $\alg A$, $\alg{Con}(\alg A)$, is given by the maps:
$$F \mapsto \theta_{F} = \{(x, y) \in A \times A : x \back y, y \back x \in F\} = \{(x, y) : x / y, y / x \in F\},$$
$$\theta \mapsto F_{\theta} = \{x \in A:  (x, 1) \in \theta\}$$
for all $F \in Fil(\alg A), \theta \in Con(\alg A)$. 
In what follows, given a congruence filter $F$, we will write $[x]_F$ for the equivalence class $[x]_{\theta_F}$.

\section{Gluing over a filter}\label{sec:gluingfilter}
As we mentioned in the introduction, the usual notion of gluing in lattice theory puts together two lattices that intersect in a filter of the first and an ideal of the second. As we are interested in integral residuated lattices and we want the components to be subalgebras of the resulting structure (in particular the common identity element needs to be the top), this approach needs to be modified. We start by describing the simple case where the ideal is empty, and the filter is trivial.

\subsection{$1$-sum}
The {\em $1$-sum} construction in the context of residuated structures was introduced with the name of {\em ordinal sum} by Ferreirim in \cite{FePhD} in the context of hoops. The latter can be defined as commutative integral divisible ($x \wedge y= x(x \to y)$) residuated lattices but without the demand that joins exist; we choose to use the naming $1$-sum as in \cite{Olson08,Olson12} to avoid confusion. Indeed it is slightly different than the ordinal sum of two posets/lattices, as it identifies the top elements of the two structures. The $1$-sum construction has played an important role in the study of BL-algebras and basic hoops \cite{AM}. The construction was later extended to integral residuated lattices, and even generalized to non-integral structures \cite{Galatos05}. It is also worth mentioning that historically, the analogue to the $1$-sum has been previously introduced and studied for semigroups \cite{C54}.

$1$-sums represent a seminal example of gluing; unlike the case of hoops, some care is needed to make sure that joins that were equal to the top still exist in the resulting structure. The two structures glued together intersect only at their respective top elements. 
In detail, let $\alg B$ and $ \alg C$ be integral residuated lattices, where $B \cap C = \{1\}$, and $1$ is join irreducible in $\alg B$ or $\alg C$ has a bottom element. We extend the order of  $\alg B$ and $\alg C$ to the set $B \cup C$ by: $b < c$ for $b \in B - \{1\}$ and $c \in C$, and extend the operations of $\alg B$ and $\alg C$ by:
$$
\begin{array}{lll}
x  y &=&\left\{
\begin{array}{ll}
y & \mbox{ if } x \in C \mbox{ and } y \in B \setminus\{1\} \\
x & \mbox{ if } x \in B \setminus\{1\} \mbox{ and } y \in C ,
\end{array}
\right.\\
&&\\
x \back y &=&\left\{
\begin{array}{ll}
y & \mbox{ if }x \in C \mbox{ and } y \in B \setminus\{1\}\\
1 & \mbox{ if } x \in B \setminus\{1\} \mbox{ and } y \in C,
\end{array}
\right.\\
&&\\
x / y &=&\left\{
\begin{array}{ll}
y & \mbox{ if }x \in C \mbox{ and } y \in B \setminus\{1\}\\
1 & \mbox{ if } x \in B \setminus\{1\} \mbox{ and } y \in C.
\end{array}
\right.\\
\end{array}\\
$$
It can be easily verified that the resulting structure $\alg B \oplus_{1} \alg C$ is a residuated lattice, called the \emph{$1$-sum} of $\alg B$ and $\alg C$.
The assumption that $1$ is join irreducible in $\alg B$ or $\alg C$ has a bottom element ensures that joins of elements of $B$, calculated in $\alg B \oplus_{1} \alg C$, exist; if these conditions are not satisfied the resulting structure is merely a residuated meet-semilattice. Note that $\alg C$ is always a subalgebra of  $\alg B \oplus_1 \alg C$ and $\alg B$ is a subalgebra except possibly with respect to $\lor$, in case that $1$ is not join irreducible in $\alg B$ (see \cite{JM09} for details when  $1$ is not join-irreducible). Generalizations of the $1$-sum construction to the non-integral case are discussed in \cite{GJKO}.

Notice that the $1$-sum construction stacks one IRL on top of  another one and identifies/glues their top elements; the product between elements of $\alg B$ and $\alg C$ is actually their meet in the new order. 

As it turns out, this is the only choice for defining a residuated monoidal operation when gluing two residuated lattices together with this particular lattice order, if we want $\alg B$ and $\alg C$ to be subalgebras of the new structure. 

\begin{proposition}\label{prop:ordinalsum}
Let $\alg B$ and $\alg C$ be IRLs and assume that $\alg D$ is an IRL with underlying set $B \cup C$, where $B \cap C = \{1\}$, $b < c$ for all $b \in B - \{1\}$ and $c \in C$, $C$ is a subalgebra, and $B$ is a subalgebra except possibly with respect to $\lor$. Then $\alg D$ is equal to $\alg B \oplus_1 \alg C$.
\end{proposition}
\begin{proof}
Given the assumptions, we only need to verify that for all $b \in B$ and $c \in C$, we have $cb =bc=b$. We have $cb \leq cb$, so $c \leq cb/b$. Since $cb, b \in B$ we have that $cb/b$ is an element of $B$ that is greater than some element of $C$. Therefore, $cb/b=1$, hence $b \leq cb$. By integrality we also have $cb\leq b$, so $cb=b$.
\end{proof}

\subsection{$F$-gluings: compatibility and uniqueness}
By relaxing the assumptions in Proposition~\ref{prop:ordinalsum}, we will generalize the $1$-sum construction to a more general type of gluing where the intersection of the two algebras may be a congruence filter different than $\{1\}$.

 More precisely, given IRLs $\alg B$ and $\alg C$, let $\alg D$ be some IRL with underlying set $B \cup C$, where $F:=B \cap C$ is a congruence filter of $\alg D$, $b < c<f$ for all $b \in B - F$, $c \in C-F$ and $f \in F$, $C$ is a subalgebra, and $B$ is a subalgebra except possibly with respect to $\lor$, in which case $F$ is assumed to have a bottom element. We say that $\alg D$ is a \emph{gluing over $F$}, or an \emph{$F$-gluing of $\alg B$ and $\alg C$}; see Figure~\ref{fig:gluingfilter} for the anticipated structure.
We will identify conditions on $\alg B$, $\alg C$ and $F$ that will allow us to construct $\alg D$ from these constituent parts. 

 First we describe a compatibility condition between $F$ and $\alg B$ and then we characterize the structure of the subset $B':= (B-F) \cup \{1\}$. Note that $B'$ supports a residuated lattice even when it is not a subalgebra under $\lor$ as those joins end up being equal to $1$.
 
We say that a congruence filter $F$ of an IRL $\alg B$ is \emph{compatible} with $\alg B$ if:  \begin{enumerate}
\item every element of $F$ is strictly above every element of $B-F$.
\item For all $b \in B-F$ the equivalence class $[b]_F$ has a maximum and a minimum; we define $$\sigma_F(b)=\min [b]_F,\;\; \gamma_F(b)=\max [b]_F$$ for  $b \in B-F$ and  $\sigma_F(1)=\gamma_F(1)=1$; hence $\sigma_F$ and $\gamma_F$ are maps on $B'=(B-F) \cup \{1\}$.  
\item $\sigma_F$ is \emph{absorbing}: for all $b \in B-F$, $b\sigma_F[B-F] \subseteq \sigma_F[B-F]$ and $\sigma_F[B-F]b \subseteq \sigma_F[B-F]$.
\end{enumerate}
We also say that $(\alg B, F)$ form a \emph{lower-compatible pair}. The following lemma shows that $F$-gluings contain compatible pairs and explains that the role of $\sigma$ and $\gamma$ is to capture the multiplications and divisions, respectively, by elements of $C$ that are not in the compatible pair. More importantly, it shows the uniqueness of the $F$-gluing.

\begin{lemma}\label{lemma:filtercompatible}
Given a gluing of IRLs $\alg B$ and $\alg C$ over $F$, the congruence filter $F$ is compatible with $\alg B$. Moreover,  for all $b \in B-F, c \in C-F$:
$$cb = bc = \sigma_F(b), \qquad c \back b = b/c = \gamma_F(b).$$ Therefore, the gluing of $\alg B$ and $\alg C$ over $F$ is unique when it exists.
\end{lemma}
\begin{proof}
We first show that for all $b \in B-F$ the equivalence class $[b]_F$ has a minimum. 
For all $b \in B-F$ and $c \in C-F$, we have $cb \leq cb$, so $c \leq cb/b$. By integrality we also have $cb\leq b \in B-F$, so $cb \in B-F$. Since $cb, b \in B$ we have that $cb/b$ is an element of $B$ that is greater than some element of $C$. Therefore, $cb/b \in F$. Also, since $cb\leq b$, we get $b/cb =1 \in F$, hence $[b]_F =[cb]_F$.  Moreover, given $b' \in [b]_F$, we have $b'/b \in F$, so $c \leq b'/b$, hence $bc \leq b'$.
Thus, $cb$ (and by symmetry also $bc$) is the minimum of $[b]_F$, for all $c \in C-F$; we denote this minimum by $\sigma_F(b)$. Note that since $F$-congruence classes are in particular lattice congruence classes,  $\sigma_F$ is monotone on $B-F$.

We now show that  for all $b \in B-F$ the equivalence class $[b]_F$ has a maximum. 
For any $c \in C-F$, we cannot have $c\back b \in C$, as then $c' \leq c \back b$ for some $c' \in C-F$, so $c'c \leq b$ and $c'c \leq c' \in C-F$, so $c'c \in C-F$, which would imply $b \in F$, a contradiction. So, $c \back b \in B-F$, and thus $\sigma_F( c \back b)=c (c\back b) \leq b$ and so $\sigma_F( c \back b) \leq \sigma_F( b)$. Also, by integrality we have $b \leq c \back b$, so $\sigma_F(b) \leq \sigma_F( c \back b)$. Therefore, $\sigma_F(b) = \sigma_F( c \back b)$, hence $[b]_F=[c \back b]_F$. Also, for every $b' \in [b]_F$, we have $cb' =\sigma_F(b') = \sigma_F(b) \leq b$, so $b'\ \leq c \back b$. Therefore, for any $c \in C-F$, $c \back b$ (and by symmetry $b/c$) is the maximum of $[b]_F$; we denote this element by $\gamma_F(b)$.

We now prove the last property of compatibility, i.e., that $\sigma_F$ is \emph{absorbing}: for all $b \in B-F$, $b\sigma_F[B-F] \subseteq \sigma_F[B-F]$ and $\sigma_F[B-F]b \subseteq \sigma_F[B-F]$. In particular, we show $b \sigma_F[B-F] \subseteq \sigma_F[B-F]$, as the proof of $\sigma_F[B-F] b  \subseteq \sigma_F[B-F]$ is similar. Every element of $b \sigma_F[B-F]$ is of the form $b\sigma_F(x)$ where $x \in B-F$. By the above, for any $c \in C-F$, $b\sigma_F(x) = bxc = \sigma_F(bx) \in \sigma_F[B-F]$. 

Thus we showed that the congruence filter F is compatible with $\alg B$, and also that $\sigma_F(b) = cb = bc$ and $\gamma_F(b) = c \back b = b/c$. It follows that the gluing over $F$ is unique when it exists.
\end{proof}

\subsection{Compatible triples inside compatible pairs}
Our aim is to characterize abstractly the individual components of the gluing and later use them to define a gluing construction. In particular we identify the structure of $(B-F) \cup \{1\}$; as this set is not closed under divisions, we will need to make use of partially defined operations.

\begin{definition}\label{def:partial}
	By a \emph{partial IRL} we understand a partially ordered partial algebra $(\alg A, \leq)$ in the language of residuated lattices, such that:
	\begin{enumerate}
	\item $\alg A$ is integral: $x \leq 1$ for all $x \in A$;
		\item the three axioms of RLs are satisfied whenever they can be applied, in the following sense: \begin{enumerate}
		\item $x \lor y$ is the least common upper bound of $x$ and $y$ whenever it exists, and similarly $x \land y$ is the largest lower bound whenever it exists;
		\item $x 1 = 1 x = 1$, and if $xy, (xy)z, yz, x(yz)$ are defined then $(xy)z = x(yz)$;
		\item if $xy$, $z / y$, and $x \back z$ are defined, then $xy \leq z$ if and only if $x \leq z / y$ if and only if $y \leq x \back z$.
		\end{enumerate}
		\item multiplication is order preserving when defined:  $a \leq b$ and $ac, bc$ defined implies $ac \leq bc$, and likewise for left multiplication.
		\item whenever defined, the division operations are order-preserving in the numerator and order-reversing in the denominator. That is:
		$x \leq y $ and $z \backslash x, z \backslash y$ defined implies $z \backslash x \leq z \backslash y$; $x \leq y $ and $ y \backslash z, x \backslash z$ defined implies $ y \backslash z \leq x \backslash z$; likewise for right division. 
	\end{enumerate}
\end{definition}

   For example, given an $F$-gluing of IRLs $\alg B$ and $\alg C$, the structure $\alg B'$, whose domain  is $B':=(B-F) \cup \{1\}$, is a partial IRL. 
   
   \begin{remark}
   We wish to remark that, even though the definition of a partial IRL we are using is quite general, the constructions we will define in the rest of the paper really involve partial algebras that are much closer to being lattices: joins will always be defined, and a meet $x \land y$ will always be defined except if there is no common lower bound of $x$ and $y$. Moreover, the partial IRLs considered in our constructions will actually have an underlying structure of a \emph{partial monoid} in the  stronger sense usually intended in the literature: the products $xy, (xy)z$ are defined if and only if $yz, x(yz)$ are defined, and in such case $(xy)z = x(yz)$.
   \end{remark}

   We will now characterize abstractly triples of the form $(\alg B', \sigma_F, \gamma_F)$, where we mean that $B':=(B-F) \cup \{1\}$. 
 
 A \emph{lower-compatible triple} $(\alg K, \sigma, \gamma)$ consists of 
 \begin{enumerate}
\item a partial IRL $\alg K$ with all operations defined, except for $x \back y$ and $y /x$ which are undefined if and only if $\sigma(x) \leq y$ and $x \not\leq y$,\item\label{property2UPT}  $(\sigma, \gamma)$  is a residuated pair, i.e. $\sigma(x) \leq y$ if and only if $x \leq \gamma(y)$, such that:\begin{enumerate}
\item $\sigma$ is a \emph{strong conucleus}, i.e, an interior operator such that for $x, y \neq 1$, $x \sigma(y) = \sigma(xy) = \sigma(x) y$, and $\sigma(1) = 1$.
\item $\gamma$ is a closure operator on $\alg K$, and
\item $x y, y  x \leq \sigma(x)$ for all $x, y \in K, y \neq 1$.
\end{enumerate}
\end{enumerate}

\begin{lemma}
If $F$ is a compatible congruence filter of an IRL $\alg B$, then $(\alg B', \sigma_F, \gamma_F)$ is a lower-compatible triple. 
\end{lemma}
\begin{proof} 
For readability, in this proof we will write $\sigma$ for $\sigma_{F}$, $\gamma$ for $\gamma_{F}$ and $\theta$ for $\theta_F$.  It is clear that $B' = (B - F) \cup \{1\}$ is closed under multiplication, meet and also under join except when $x \vee y \in F-\{1\}$; we redefine these joins to be $1$ in $\alg B'$. Since $\alg B$ is an IRL and $B'$ inherits its operations, it can be directly checked that $\alg B'$ is a partial IRL in the sense of Definition \ref{def:partial}. With respect to the divisions, we want to show that $x \back y$ and $y /x$ are undefined if and only if $\sigma(x) \leq y$ and $x \not\leq y$. Notice that the divisions $x \back y$ and $y /x$ are undefined in $B'$ iff they produce elements of $F -\{1\}$. From $x \back y\in F-\{1\}$ we get that $x \not\leq y$, and moreover $f \leq x \back y$ for some $f \in F-\{1\}$. Thus by residuation $xf \leq y$. Thus, since $xf \in [x]_{F}$ (because $xf \leq x$ and $f \leq x \back xf$), we have $\sigma(x) = \min[x]_{F} \leq xf \leq y$. Similarly we can prove that if $y /x$ is not defined in $B'$ then again $\sigma(x)\leq y$ and $x \not\leq y$. Conversely, suppose $\sigma(x) \leq y$ and  $x \not\leq y$. Then $x \back \sigma(x) \leq x \back y$ and since $x \back \sigma(x) \in F$, we get $x \back y \in F$. Moreover, since $x \not\leq y$, we get $x \back y \not= 1$. 

Note that $B'$ is closed under meet as all elements of $B-F$ are below all elements of $F$ and it is closed under multiplication due to integrality and order preservation of multiplication. Also, it is closed under joins that do not produce elements of $F-\{1\}$ and the ones that do produce such elements are redefined to be $1$. The resulting structure is a monoid and a lattice. Finally, if $x \back^{\alg B} y \not \in F-\{1\}$, then residuation holds as all terms are evaluated in $B'$.

We now prove that $\sigma$ is a strong conucleus. Clearly, $\sigma(x) \leq x$ and $\sigma(\sigma(x)) = \sigma(x)$ thus $\sigma$ is decreasing and idempotent.
We now prove that $\sigma$ is monotone. Suppose $x \leq y$, with $x, y \in B'$. Since $\sigma(x) \mathrel{\theta} x$ and $\sigma(y) \mathrel{\theta} y$, we have $\sigma(x) \land \sigma(y) \mathrel{\theta} x \land y = x \mathrel{\theta} \sigma(x)$. Thus $\sigma(x) \leq \sigma(x) \land \sigma(y)$ (since $\sigma(x)$ is the smallest element in the equivalence class), thus $\sigma(x) \leq \sigma(y)$. We will now use the absorbing property of $\sigma$ in order to show that it is a strong conucleus. We show that $x\sigma(y) = \sigma(xy) = \sigma(x)y$ for $x$ and $y$ not equal to $1$. Now, $x \sigma(y) \in x \sigma[B-F] \subseteq \sigma[B-F]$, thus $x \sigma(y) = \sigma(z)$ for some $z \in B-F$. But then since $\sigma$ is idempotent $\sigma(x \sigma(y)) = \sigma(\sigma(z)) = \sigma(z) = x \sigma(y)$. Since $\sigma$ is decreasing and order preserving, we get $x \sigma(y) = \sigma(x \sigma(y)) \leq \sigma(xy)$. Moreover, since $x \mathrel{\theta} x$ and $y \mathrel{\theta} \sigma(y)$, we get $x y \mathrel{\theta} x \sigma(y)$, thus $\sigma(xy) = \sigma(x \sigma(y)) \leq x \sigma(y)$, and this shows that $\sigma(xy) = x \sigma(y)$.
Similarly, using $\sigma[B-F] x \subseteq \sigma[B-F]$, we can show that $ \sigma(y) x = \sigma(xy)$. 

We now prove that $\gamma$ is a closure operator. Since $\gamma(x) = \max[x]_{F}$, it is easy to see that it is increasing and idempotent. 
To show that $\gamma$ is monotone, suppose $x \leq y$, with $x, y \in B-F$. Since $\gamma(x) \mathrel{\theta} x$ and $\gamma(y)\mathrel{\theta} y$, we have $\gamma(x)\lor \gamma(y) \mathrel{\theta} x \lor y = y \mathrel{\theta} \gamma(y)$. Thus $\gamma(x)\lor \gamma(y)  \leq \gamma(y)$ (since $\gamma(y)$ is the biggest in the equivalence class),  thus $\gamma(x) \leq \gamma(y)$.

We now show that $(\sigma, \gamma)$ is a residuated pair. If $\sigma(x) \leq y$ then by idempotency and monotonicity of $\sigma$ we get that $\sigma(x) = \sigma\sigma(x) \leq \sigma(y)$. Then $\gamma(\sigma(x)) \leq \gamma(\sigma(y))$, and considering that it follows from their definition that $\gamma \circ \sigma = \gamma$, we have $x \leq \gamma(x) \leq \gamma(y)$. Similarly, if $x \leq \gamma(y)$ then $\gamma(x) \leq \gamma(y)$, thus applying $\sigma$ we get that $\sigma(x) \leq \sigma(y) \leq y$.

It is only left to prove that $x y, y  x \leq \sigma(x)$ for all $x, y \in B', y \neq 1$. This easily follows from residuation, since for instance: $x y \leq \sigma(x)$ if and only if $y \leq x \back \sigma(x)$, which holds since $x \back \sigma(x) \in F$.
\end{proof}
We say that $(\alg B', \sigma_F, \gamma_F)$ is the compatible triple of the compatible pair $(\alg B, F)$. We can also show that every compatible triple comes from a compatible pair. 

\begin{lemma}
Every lower-compatible triple $(\alg K, \sigma, \gamma)$ is the compatible triple of the lower-compatible pair $(\alg B, G)$, where $G$ is the 2-element IRL and  $\alg B$ is an IRL with operations extending $(\alg K-\{1\})\cup \alg G$.
\end{lemma}
\begin{proof} 
Let $B = (K-\{1\})\cup G$ where $G = \{f, 1\}$.
We extend the operations of $\alg K$ to $\alg B$ except when $x \lor y = 1$ in $K$, in which case we redefine $x \lor^{\alg B} y = f$; moreover we stipulate that:
\begin{itemize}
\item $f$ is an idempotent coatom strictly above all elements of $K - \{1\}$;
\item $f\cdot x=x\cdot  f=\sigma(x)$, $f \back  x = x/ f = \gamma(x)$, $x \back  f = f /  x =1$ for all $x \in K - \{1\}$.
\item $x \back  y = y /x = f$ for $x, y \in B$ such that $x \not\leq y$ and $\sigma(x) \leq y$ (i.e., when $x \back^{\alg K}y, y/^{\alg K}x$ are undefined).
\end{itemize}

We first show that $\alg B$ is a residuated lattice. It can be easily seen that $(B, \land , \lor , 1)$ is a lattice with top $1$. In order to see that $(B , \cdot , 1)$ is a monoid, we need to prove associativity of the product in triples of elements where $f$ is involved, as in the other cases associativity follows from the associativity in $\alg K$. We will make use of the absorption of $\sigma$. For example, if $b, d \in K- \{1\}$ (the other cases are similar): $$b(df) = b \sigma(d) = \sigma(bd) = (bd)f,$$ $$b(fd) = b \sigma(d) = \sigma(bd) = \sigma(b) d = (bf)d.$$
For residuation, notice first that if one among $x, y, z$ is $1$ then the law clearly holds. First we check the cases where $f$ is involved. For $b, d \in K - \{1\}$, we show that:
$$b  f \leq  d \mbox{ iff } f \leq  b \back  d \mbox{ iff } b \leq  d/  f. $$
Indeed, $b f \leq  d \mbox{ iff } \sigma(b) \leq  d $ iff  [$b \leq d$ or ($b \not\leq d$ and $\sigma(b) \leq d$)] iff  ($b \back  d= 1$ or $b \back  d = f$) iff $f \leq  b \back  d$. Moreover, $b f \leq  d \mbox{ iff } \sigma(b) \leq d$ iff $b \leq \gamma(d)$ iff $b \leq  d/  f$, where we used that $(\sigma, \gamma)$ is a residuated pair. 
Likewise we obtain $f b \leq  d \mbox{ iff } b \leq  f \back  d \mbox{ iff } f \leq  d/  b. $
 Also,
$$b  d \leq  f \mbox{ iff } d \leq  b \back  f \mbox{ iff } b \leq  f/  d $$
holds as all these statements are true even for $b=f$ and $d \in K-\{1\}$ and also for  $d=f$ and $b \in K-\{1\}$. Moreover, $ff \leq b$ iff $f \leq b/f$ iff $f \leq f \back b$ since all inequalities are false. Now for $x, y, z \in K -  \{1\}$, we want to show that $$x y \leq  z \mbox{ iff } y \leq  x \back  z \mbox{ iff } x \leq  z /  y.$$
We show that $x y \leq  z \mbox{ iff } y \leq  x \back  z$, as the other equivalence is obtained similarly. We consider three different cases.
\begin{itemize}
\item If $x \leq z$, then both inequalities are true, since we get $xy \leq x \leq z$ and $x \back  z = 1$ thus $y \leq x \back z$.
\item If $x \not\leq z$, and $\sigma(x) \leq z$, we have to show that $xy \leq z$ iff $y \leq f$. Notice that $y \leq f$ since $y \neq 1$. Moreover, since $y \neq 1$ then $xy \leq \sigma(x) \leq z$.
\item If $\sigma(x) \not \leq z$, then the operations are defined in $K$ and residuation holds.
\end{itemize}
Thus $\alg B$ is an integral residuated lattice. 

We want to show that $(\alg B, G)$ is a compatible pair. $G$ is closed under products since $f$ is idempotent and it is closed under conjugates since $(x\cdot f)/ x = \sigma(x) / x \in \{f, 1\}$ and $x \back (f\cdot x) = x \back \sigma(x) \in \{f, 1\}$; so $G$ is a congruence filter of $B$.

We now show that $\sigma (x) = \min[x]_{G} = \sigma_{G}(x)$ and $\gamma (x) = \max[x]_{G} = \gamma_{G}(x)$ for all $x \in  K -  \{1\}$. 
 Note that $\sigma(x) \leq x$ implies $\sigma(x) \back  x = 1 \in G$; also $x \back  \sigma(x)$ is either $1$ or $f$, thus still in $G$. Therefore, $\sigma(x) \in [x]_{G}$. Furthermore, if $y \in [x]_{G}$, then $x \back  y \in G$, thus either $x \back  y = 1$ or $x \back  y = f$. If $x \back  y = 1$  then $\sigma(x) \leq x \leq y$, while if $x \back  y = f$, then $xf \leq y$, so $\sigma(x) \leq y$. Thus in any case $\sigma(x) \leq y$, which implies that $\sigma(x) = \min[x]_{G} = \sigma_{G}(x)$.

We now prove that $\gamma (x) = \gamma_{G}(x)$ for all $x \in K - \{1\}$. Note that $x\leq \gamma(x)$ implies $x \back  \gamma(x) = 1 \in G$; also $\gamma(x) \back  x$ is either $1$ if $x = \gamma(x)$ or $f$ otherwise, since $\sigma(\gamma(x)) \leq x$ follows from the fact that $\sigma, \gamma$ form a residuated pair. Therefore, $\gamma(x) \in [x]_{G}$. Now, if $y \in [x]_{G}$, then $y \back  x \in G$, thus either $y \back  x = 1$ or $y \back  x = f$. If $y \back  x = 1$  then $y \leq x \leq \gamma(x)$, while if $y \back  x = f$, then $yf \leq x$, so $\sigma(y) \leq x$, if and only if $y \leq \gamma(x)$. Thus $\gamma(x) = \max[x]_{G} = \gamma_{G}(x)$.

To prove that $\sigma$ is absorbing, we use the fact that $\sigma$ is a strong conucleus. For any $x$ and $y \in B-G$, we have $x \sigma(y)=\sigma(xy)\in \sigma[B-G]$, so $x \sigma[B-G] \subseteq \sigma[B-G]$ and similarly $\sigma[B-G] x \subseteq \sigma[B-G]$. Thus we proved that $(\alg B, G)$ is a compatible pair, and since $\sigma = \sigma_{G}, \gamma = \gamma_{G}$, and $K = (B-F) \cup \{1\}$, it follows that $(\alg K, \sigma, \gamma)$ is its compatible triple.
\end{proof}
If $\alg 2$ denotes the two-element residuated lattice, we have also shown the following.  
\begin{corollary}\label{c:multcomp}
 If $(\alg B, F)$ is a lower-compatible pair, then so is $(\alg B_F, \mathbf{2})$, where $B_F=(B-F) \cup 2$.  
 \end{corollary}
We say that an IRL $\alg F$ \emph{fits} with a lower-compatible triple $(\alg K, \sigma, \gamma)$, if $(\alg K-\{1\})\cup \alg F$ extends to an IRL $\alg B$, $F$ is a compatible filter of $\alg B$, and the compatible triple of the compatible pair $(\alg B, F)$ is $(\alg K, \sigma, \gamma)$.  Note that if $(\alg B, F)$ is a compatible pair, then $F$ fits with the compatible triple $(\alg B', \sigma_F, \gamma_F)$, but Corollary~\ref{c:multcomp} shows that the same $\alg B'$ can belong to different lower-compatible triples. 
\subsection{Fitting the components together: the construction}

Now, we define the $F$-gluing construction given the individual pieces we have identified, provided that they fit together in a suitable way.

Consider a lower-compatible pair $(\alg B, F)$, and an IRL $\alg C$ such that $B \cap C = F$, with $F$ strictly above all other elements in $C$ and such that if there are elements in $B-F$ joining to some element of $F$ then $C$ has a least element $0_{C}$. We will show that there is an IRL that is the (unique) gluing of $\alg B$ and $\alg C$ over $F$ and that it is the following structure:
 $$\alg B  \oplus_{F} \alg C= (B \cup C,\, \cdot_{F},\back_{F},/_{F},\land_{F}, \lor_{F}, 0,  1),$$
where the operations are defined as follows:
\begin{align*}
x \cdot_{F} y &=\left\{\begin{array}{ll}
x \cdot y & \mbox{ if } x, y \in  B, \mbox{ or } x, y \in  C\\
\sigma_{F}(x) & \mbox{ if } y \in C-F, x \in B-F\\
\sigma_{F}(y) & \mbox{ if } x \in C-F, y \in B-F
\end{array}\right.\\
x\back_{F}\, y &=\left\{\begin{array}{ll} 
x \back y & \mbox{ if } x, y \in  B, \mbox{ or } x, y \in  C\\
\gamma_{F}(y) & \mbox{ if } x \in C-F, y \in B-F \\
1 & \mbox{ if } x \in B-F, y \in C-F \end{array}\right.\\
x/_{F}\, y &=\left\{\begin{array}{ll} 
x / y & \mbox{ if } x, y \in  B, \mbox{ or } x, y \in  C\\
\gamma_{F}(x) & \mbox{ if } y \in C-F, x \in B-F \\
1 & \mbox{ if } x \in C-F, y \in B-F \end{array}\right.\\
x \land_{F} y &=\left\{\begin{array}{ll}
x \land y & \mbox{ if } x, y \in  B, \mbox{ or } x, y \in  C\\
x & \mbox{ if } x \in B-F, y \in C-F\\
y & \mbox{ if } y \in B-F, x \in C-F\end{array}\right.\\
x \lor_{F} y &=\left\{\begin{array}{ll}
x \lor y & \mbox{ if } x, y \in  C, \mbox{ or } x, y \in  B \mbox{ with } x \lor y \not\in F\\
0_{C} & \mbox{ if } x, y \in  B-F,  x \lor y \in F\\
y & \mbox{ if } x \in B-F, y \in C-F\\
x & \mbox{ if } y \in B-F, x \in C-F\end{array}\right.
\end{align*}
 The proof of the following theorem is postponed until the next section where we further expand the construction. More precisely, it will be a direct consequence of Theorem \ref{prop:quadruplegluing}.
\begin{figure}
\begin{center}
\begin{tikzpicture}
\footnotesize{
 \draw  [dashed](0,0) ellipse (0.3 and 0.4);
 \draw  (0,1) ellipse (0.3 and 0.4);
  \fill (0,1.4) circle (0.05);
 \node at (3,1) {$F$}; 
  \node at (0,1) {$F$}; 
   \node at (6,1.9) {$F$}; 
 \node at (0.15,1.6) {$1$}; 
  \node at (0,-1) {$\alg B$};
  \node at (1.5,0.5) {$\oplus_{F}$};
  \draw [densely dotted] (3,0) ellipse (0.3 and 0.4);
  \draw  (3,1) ellipse (0.3 and 0.4);
  \fill (3,1.4) circle (0.05);
 \node at (3.15,1.6) {$1$}; 
  \node at (3,-1) {$\alg C$};
   \node at (4.5,0.5) {${=}$};
    \draw  [dashed](6,-0.05) ellipse (0.3 and 0.4); 
    \draw [densely dotted] (6,0.95) ellipse (0.3 and 0.4);
 \node at (6.15,2.5) {$1$}; 
 \draw  (6,1.9) ellipse (0.3 and 0.4);
  \node at (6,-1) {$\alg B \oplus_{F} \alg C$};
\fill (6,2.3) circle (0.05);  
   }
 \end{tikzpicture} \caption{The gluing $\alg B \oplus_{F} \alg C$ of the algebras $\alg B$ and $\alg C$ over $F$.} \label{fig:gluingfilter}
   \end{center}  
\end{figure}
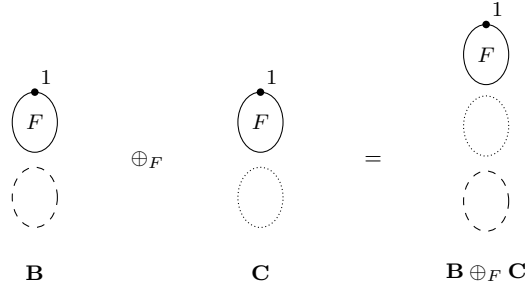
\begin{theorem}\label{prop:gluingfilter}
If $F$ is a congruence filter of an IRL $\alg C$ and also a compatible congruence filter of an IRL $\alg B$, then $\alg B \oplus _F \alg C$ is the gluing of $\alg B$ and $\alg C$ over $F$. 
\end{theorem}
\subsection{Gluing without the filter}
We now obtain a different construction, that generalizes the $1$-sum construction in a different way: it glues together two structures that intersect at the top $1$ and maintains the same order relation, but some of the divisions are redefined. With respect to the previous construction, the intuition here is that we are removing the filter $F$ (keeping the unit $1$), and what is left is a partial IRL.

We start from a lower compatible triple $(\alg K, \sigma, \gamma)$ and a partial IRL $\alg L$ where some of the divisions might not be defined. In particular, $x \back y$ is undefined if and only if all elements $z$ in the interval $[y, 1) = \{z \in L: y \leq z < 1\}$ are such that $xz \leq y$ but $[y, 1)$ does not have a top element. Similarly, $y /x$ is undefined if and only if all elements $z$ the interval $[y, 1)$ are such that $zx \leq y$ and $[y,1)$ has no top. If $\alg L$ has a splitting coatom $c_L$ (i.e., $L=\{1\} \cup {\downarrow} c_L$), all divisions in $L$ are defined.  We also assume that  $K \cap L = \{1\}$ and if $x \lor y = 1$ in $K$ for some $x,y \in K-\{1\}$, then $L$ has a bottom element $0_{L}$.

We set $\pi = (\sigma, \gamma)$ and we define $\alg K \oplus_{\pi}  \alg L$ to be the structure where the operations extend those of $K$ and $L$, except if $x \lor y = 1$ in $K$ then we redefine $x \lor y = 0_{L}$. Moreover:
\begin{align*}
x y &=\left\{\begin{array}{ll}
\sigma(x) & \mbox{ if } y \in L - \{1\}, x \in K - \{1\}\\
\sigma(y) & \mbox{ if } x \in L - \{1\}, y \in K - \{1\}\end{array}\right.\\
x\back\, y &=\left\{\begin{array}{ll} 
c_{L} & \mbox{ if } x, y \in  K, L \mbox{ has a splitting coatom } c_{L}, \mbox { and $x \back^{\alg K} y$ is undefined}\\ 
\gamma(y) & \mbox{ if } x \in L - \{1\}, y \in K - \{1\} \\
1 & \mbox{ if } x \in K - \{1\}, y \in L\\ \end{array}\right.\\
y/\, x &=\left\{\begin{array}{ll} 
c_{L} & \mbox{ if } x, y \in  K, L \mbox{ has a splitting coatom } c_{L}, \mbox{ and $y/^{\alg K} x$ is undefined}\\
\gamma(y) & \mbox{ if } x \in L - \{1\}, y \in K - \{1\} \\
1 & \mbox{ if } x \in K - \{1\}, y \in L\\\end{array}\right.\\
x \land y &=\begin{array}{ll}
x & \mbox{ if } x \in K - \{1\}, y \in L\end{array}.\\
x \lor y &=\begin{array}{ll}
y & \mbox{ if } x \in K -\{1\}, y \in L\end{array}.
\end{align*}

\begin{proposition}\label{prop:partial1}
$\alg K \oplus_{\pi}  \alg L$ is a partial IRL and it is total if $L$ has a splitting coatom.
\end{proposition}
 
 \begin{proof}
 First, we notice that all operations are total except possibly the divisions.
 The meet and join operation clearly define a lattice, and $x \leq 1$ for every element $x$ in the gluing. 
Associativity of the product can be easily shown by the definition using the idempotency of $\sigma$ and the strong conuclearity condition. Since the operations of $\alg K \oplus_{\pi}  \alg L$ extend the ones of $\alg K$ and $\alg L$, $\alg K \oplus_{\pi}  \alg L$ has an underlying monoidal structure. 

We show that residuation holds. That is, for all elements $x, y, z$, whenever $xy, x\back z, z / y$ are defined: $$xy \leq z \mbox{ iff } y \leq x \back z \mbox{ iff } x \leq z /y$$

\begin{itemize}
\item In the case where $x, y \in L, z \in K-\{1\}$ none of the inequalities are true, as can be seen by the definition of the order and the divisions.
\item In the case where $x, z \in L, y \in K-\{1\}$ all the inequalities hold as it follows from the definition of the order and of the divisions. The situation is similar when: $x \in K- \{1\}, y, z \in L$, as well as when $x, y \in K-\{1\}, z \in L$.
\item If $x \in L-\{1\}, y,z, \in K-\{1\}$, we need to verify: $$\sigma(y) \leq z \mbox{ iff } y \leq \gamma(z) \mbox{ iff } x \leq z/ y$$ whenever  $z / y$ is defined.
The first equivalence holds since the two maps form a residuated pair. We now show $\sigma(y) \leq z \Leftrightarrow x \leq z/ y$ holds in case $z/ y$ is defined in $\alg K \oplus_{\pi}  \alg L$.
If $y \leq z$, then $\sigma(y)\leq \sigma(z) \leq z$ and $x \leq 1 = z /y$, so both inequalities hold. Assume now that $y \not \leq z$. If  $\sigma(y) \leq z$ we get that $z/y$ is undefined in $\alg K$, thus $z/ y$ is the coatom of $L$ (supposing $z/ y$ is defined in $\alg K \oplus_{\pi}  \alg L$), so $x \leq z/ y$ holds. Conversely, if $x \leq z/ y$ then since $x \in L$, we get that  either $ z/y = 1$ or it is the coatom of $L$. Thus either $y \leq z$ (which is against our assumption) or $z/y$ is undefined in $K$, which means that $\sigma(y) \leq z$. 
The case: $x, z \in K-\{1\}, y \in L-\{1\}$ is similar.

\item Suppose now $x, y, z \in K-\{1\}$. We only show  $x y \leq z \Leftrightarrow y \leq x \back z$ assuming that $x \backslash z$ is defined in $\alg K \oplus_{\pi}  \alg L$;  the proof of the equivalence  $x y \leq z \Leftrightarrow x\leq z / y$ is similar.
If  $x \back z$ is defined in $K$, the equivalence holds by residuation in $\alg K$. Also, if $x \leq z$, then $x y \leq zy \leq z$ and $y \leq 1= x \back z$, so both inequalities hold. Now, if $x \back^{\alg K} z$ is undefined and  $\sigma(x) \leq z$, then we get $xy \leq \sigma(x) \leq z$ (where the first inequality is due to (\ref{property2UPT}c) in the definition of lower-compatible triple) and also that $x \back z$ is equal to the coatom of $L$, so again both inequalities hold. 
\end{itemize}

It is left to show that multiplication is order preserving, divisions are order preserving in the numerator and order reversing in the denominator. 
The fact that the monoidal operation is order preserving can be easily checked using the facts that by the definition of lower-compatible triple $\sigma$ is order preserving, and $xy, yx \leq \sigma(x)$ for all $x,y \in K, y \neq 1$. Finally, the order properties of divisions (when defined) can be directly checked, and follow by residuation and the order preservation of $\gamma$.

We have shown that $\alg K \oplus_{\pi}  \alg L$ is a partial IRL. In the case where $L$ has a coatom, all divisions (the only partial operations) are correctly and fully defined. Thus in that case $\alg K \oplus_{\pi}  \alg L$ is a (total) IRL.
 \end{proof}
We will refer to $\alg K \oplus_{\pi}  \alg L$ as the \emph{partial upper gluing} of $\alg K$ and $\alg L$. If we take a lower-compatible pair $(\alg B, F)$ and an IRL $\alg C$ such that $B \cap C = F$ with $F$ strictly above the other elements in $C$, we can construct the partial gluing of the compatible triple $(\alg B, \sigma_{F}, \gamma_{F})$ and of $C-F$, which is a partial IRL with the required properties for the partial gluing construction.
We will see interesting examples of these constructions in the final section of the paper.
\section{Gluing over a filter and an ideal}

We can take the previous intuition even further and generalize the construction allowing the algebras $\alg B$ and $\alg C$ to intersect on both a congruence filter $F$ and a lattice ideal $I$. Let $\alg D$ be an IRL with underlying set $B \cup C$, where $B \cap C = F \cup I$, with $F$ a congruence filter as before, $I$ a lattice ideal, $i < b < c < f$ for all $i \in I$, $b \in B^{-}:= B - (F \cup I)$ $c \in C^{-}:= C - (F \cup I)$ and $f \in F$, $C$ is a subalgebra except possibly for $\land$ and $B$ is a subalgebra except possibly with respect to $\lor$. We say that $\alg D$ is \emph{a gluing of $\alg B$ and $\alg C$ over $F$ and $I$}, or \emph{$F-I$-gluing of $\alg B$ and $\alg C$}.
As before, we will identify conditions on $\alg B$, $\alg C$, $F$ and $I$ that will allow us to construct $\alg D$ from these constituent parts.  We will characterize the structure on the subset $C'=(C-I)$.

\subsection{Compatibility with an ideal}
 
We call an element $c \in C^{-}$ a \emph{left $I$-divisor} if there exists another $c' \in C^{-}$ such that $cc' \in I$, and a \emph{right $I$-divisor} if instead there exists $c'' \in C^{-}$ such that $c''c \in I$.

We say that an ideal $I$ of an IRL $\alg C$ is \emph{compatible} with $\alg C$ if it is strictly below $C'$ and for all left $I$-divisors $c \in C'$ and right $I$-divisors $d \in C'$, the sets $\{c \back i : i \in I\}$ and $\{i/ d : i \in I\}$ have maxima. In this case we denote these
 elements by $\ell_I(c)$ and $r_I(d)$, respectively:
 $$\ell_I(c) = {\rm max}\{c \back i : i \in I\}, \;\; r_I(d) = {\rm max} \{i/ d : i \in I\}.$$
We also say that $(\alg C, I)$ is a \emph{upper-compatible pair}.

\begin{lemma} \label{lemma:opideal}
Given a gluing of $\alg B$ and $\alg C$ over $F$ and $I$, the lattice ideal $I$ is compatible with $\alg C$. Also, for every  left $I$-divisor $c \in C-I$, right $I$-divisor $d \in C-I$, and $b \in B-F$:
$$c \back b = \ell_{I}(c), \qquad b/d = r_{I}(d).$$
\end{lemma}
\begin{proof} 
The ideal $I$ is strictly below all elements of $C' = C -I$ by definition. We now show that the maps $\ell_{I}$ and $r_{I}$ are defined for all, respectively, left and right $I$-divisors.
For a left $I$-divisor $c \in C^-$ we consider $c \back b$ for some $b \in B^{-}$. We claim that $c \back b \in C^{-}$. Indeed by definition there is at least one element in $C^{-}$ that multiplied to the right of $c$ gives an element of $I$, thus below $b$, and moreover there cannot be $f \in F$ such that $cf \in I$. Indeed otherwise, we would have $cf \leq i$ for some $i \in I$ iff (by residuation) $c \leq i/f$, which yields $i/f \in F$ (since $i/f \in B \cap C$); 
this would imply that $f'f \leq i$ for some $f' \in F$, so we would get $i  \in F$, a contradiction. 
 Thus $c \back b = \max\{d \in C^{-}: c  d \in I\}$. 

We now show that $\max\{d \in C^{-}: c d \in I\} = \max\{c \back i : i \in I\}$. Indeed, since $c \back b = \max\{d \in C^{-}: c  d \in I\}$, there exists $j \in I$, with $c  (c \back b) \leq j$; thus $c \back b \leq c \back j$. Also, for all $i \in I$, we have $c \back i \leq c\back b$ since $i <b$; so $c \back i \leq c\back b \leq c \back j$. Since $j \in I$, we get that $c \back b = c \back j$ and also $c \back b =  \max\{c \back i : i \in I\}$. Thus $\max\{c \back i : i \in I\}$ exists for every left $I$-divisor $c$ in  $C'$ and $c \back b=\ell_I(c)$. Similarly, one can show that $\max\{i/ d : i \in I\}$ exists for every right $I$-divisor $d$ of $C'$.
\end{proof}

What we showed in the previous section about lower-compatible pairs still holds in case there is at least one non $I$-divisor. Otherwise, we need consider the weaker notion of lower-compatibility.
\begin{definition}
	We say that a pair $(\alg B, F)$ is a \emph{weak lower-compatible pair} if it respects all the conditions of a lower-compatible pair except that $\gamma_F$ need not be defined, i.e. the equivalence classes $[b]_F$ for $b \in B-F$ do not need to have a maximum element.
\end{definition}
We can indeed prove the following analogue of Lemma \ref{lemma:filtercompatible}.
\begin{lemma}\label{lemma:filtercompatible2}
	Given a gluing of $\alg B$ and $\alg C$ over $F$ and $I$, $(\alg B, F)$ is a weak lower-compatible pair. If there exists an element of $\alg C$ that is not an $I$-divisor, then $(\alg B, F)$ is a lower-compatible pair. Moreover, for all $b \in B-F, c \in C-(F \cup I)$:
$$cb = bc = \sigma_{F}(b),$$  
		$$c \back b = \gamma_{F}(b) \mbox{ if c is not a left $I$-divisor, }$$	
	$$b / c = \gamma_{F}(b) \mbox{ if c is not a right $I$-divisor. }$$
\end{lemma}
\begin{proof}
	The proof that in any case $(\alg B, F)$ is a weak lower compatible triple is the same as the one of Lemma \ref{lemma:filtercompatible}. If given some $c \in C - (F \cup I)$, there is no element $c' \in C-(F \cup I)$ such that $cc' \in I$, then also the proof about $[b]_F$ having a maximum can be replicated in exactly the same way. Indeed, we get  that then $c \back b \in B$, $[b]_F = [c \back b]_F$ and $b' \leq c \back b$ for all $b' \in [b]_F$.
\end{proof}

\subsection{Compatible quadruples}

Consider now a weak lower-compatible pair $(\alg B, F)$, and an upper-compatible pair $(\alg C, I)$ such that $B \cap C = F \cup I$ is a subalgebra of both $\alg B$ and $\alg C$, with $F$ strictly above all elements in $C-F$ and $I$ strictly below all elements of $C-I$. 
We call the quadruple $(\alg B, F, \alg C, I)$ \emph{compatible} if moreover: 
\begin{enumerate}
\item\label{compat-cond-1} if at least an element of $C-F$ is not an $I$-divisor, then $(\alg B, F)$ is a lower-compatible pair.
\item\label{compat-cond-ideal} Whenever $c, d \in C-I$ and $cd \in I$, we have $(cd) x = x (cd) = \sigma_F(x)$ for all $x \in B^-$.
\item\label{compat-cond-2} If there are elements $x, y \in B-F$ such that $x \lor y \in F$, then $C-I$ has a bottom element $\bot_{C}$.
\item\label{compat-cond-3} If there are elements $x, y \in C-I$ such that $x \land y \in I$, then $B-F$ has a top element $\top_{B}$.
\end{enumerate}

We recall that $\sigma_{F}(x) = \min[x]_{F},$ and whenever defined $\gamma_{F}(x) = \max[x]_{F}$, $\ell_{I}(x) = \max\{x \back i : i \in I\},\; r_{I}(x) = \max\{i/x: i \in I\}$.

\begin{proposition}
Given a gluing of IRLs $\alg B$ and $\alg C$ over $F$ and $I$, $(\alg B, F, \alg C, I)$ is a compatible quadruple. Moreover, the gluing of $\alg B$ and $\alg C$ over $F$ and $I$ is unique when it exists.
\end{proposition}

\begin{proof}
The fact that $(\alg B, F)$ is a weak lower compatible pair, together with conditions \ref{compat-cond-1} and \ref{compat-cond-ideal}, are shown in Lemma \ref{lemma:filtercompatible2}. The fact that $(\alg C, I)$ is an upper-compatible pair is shown in Lemma \ref{lemma:opideal}. Conditions \ref{compat-cond-2} and \ref{compat-cond-3} are clearly properties of the lattice ordering of the gluing.
\end{proof}
The following technical properties of a compatible quadruple will be useful in what follows.
\begin{lemma}\label{bounded-lemma} If $(\alg B, F, \alg C, I)$ is a compatible quadruple, the following properties hold.
\begin{enumerate}
\item\label{conuclei-abs-b}
 For all $x \in F, y \in B - F$, $x \sigma_{F}(y) = \sigma_{F}(x y) = \sigma_{F}(y x) = \sigma_{F}(y) x = \sigma_{F}(y)$.
\item\label{condC-} For every $c \in C^{-}$ and $f \in F$, we have $cf, fc \in C^{-}$.
\end{enumerate}
\end{lemma}

\begin{proof}
(1) Since $x \mathrel{\theta} 1$ and $y \mathrel{\theta} y$, we have $xy \mathrel{\theta} y$ which implies $\sigma_{F}(xy) = \sigma_{F}(y)$. 

From $x \mathrel{\theta} x$ and $y \mathrel{\theta} \sigma_{F}(y)$ we get $xy \mathrel{\theta} x \sigma_{F}(y)$, thus $\sigma_{F}(xy) = \sigma_{F}(x \sigma_{F}(y)) \leq x \sigma_{F}(y)$. Moreover, from $x \mathrel{\theta} 1$ and $y \mathrel{\theta} \sigma_{F}(y)$ we have $xy \mathrel{\theta} \sigma_{F}(y)$, which implies $\sigma_{F}(xy) = \sigma_{F}(\sigma_{F}(y)) = \sigma_{F}(y) \geq x \sigma_{F}(y)$; thus $\sigma_{F}(xy) = x \sigma_{F}(y)$. The other equalities can be proven analogously.

(2) Suppose by way of contradiction that $cf = i \in I$. Thus by residuation $c \leq i/f$, but since $i, f \in B \cap C$, and $ B \cap C$ is a subalgebra of both $\alg B$ and $\alg C$, we get $i/f \in F$, so $i \geq (i/f)f \in F$, hence $i \in F$, a contradiction. Similarly one can show that $fc \in C^{-}$. 
\end{proof}

\subsection{Abstracting the upper part}

We now characterize abstractly the triples of the form $(\alg C', \ell_I, r_I)$, where we understand $C'=(C-I)$ and $I$ is an ideal strictly below $C'$. A triple $(\alg L, \ell, r)$ is called \emph{upper-compatible}, if $\alg L$ is a partial IRL and:
\begin{enumerate}
\item\label{propUC1} $ \ell, r$ are partial maps on $L$ that form a Galois connection; more precisely $\ell(y)$ is defined and $x \leq \ell(y)$ if and only if $r(x)$ is defined and $y \leq r(x)$. 
\item If $r(y')$ is defined, $x \leq r(y')$ and $y \leq y'$, then $r(y)$ is defined and $x \leq r(y)$. Thus the domain $D_{r}$ of $r$  is downwards closed. Also, the same holds for $\ell$.
\item\label{propUC3} $xy$ is undefined iff $y$ is in the domain of $r$ and $x \leq r(y)$, iff $x$ is in the domain of $\ell$ and $y \leq \ell(x)$; 
\item $x \back y$ is undefined iff there is no $z$ with $xz \leq y$, and $y / x$ is undefined iff there is no $z$ with $zx \leq y$.
\item\label{propassoc} If $\ell(x)$ and $r(z)$ are defined then: $x \back r(z)$ is defined iff $\ell (x) / z$ is defined, and in such a case $x \back r(z)=\ell (x) / z$. 
\item\label{propassoc2} If $\ell(x)$ is undefined and $r(z)$ is defined, then $x \back r(z) = r(z)$. If $r(z)$ is undefined and $\ell(y)$ is defined, then $\ell(y)/z = \ell(y)$.
\item\label{prop-orderell} If $\ell(x)$ is defined, then $x \back z$ is defined and $\ell(x) \leq x \back z$. Similarly,  if $r(x)$ is defined, then $w/x$ is defined and $r(x) \leq w/x$.
\item $x \land y$ is undefined iff there is no $z \leq x, y$. 
\item All other operations are defined.
\end{enumerate}
\begin{lemma}\label{lemma:partialres}
In an upper-compatible triple, if $x y$ is defined, then: $xy \leq z$ iff ($x\back z$ is defined and $y \leq x \back z$) iff ($z/y$ is defined and $x \leq z/y$).
\end{lemma}
\begin{proof}
Suppose that $x y$ is defined. 
If $xy \leq z$, then $x\back z$ is defined, since there is $y$ such that $xy \leq z$, and so residuation holds. Conversely, suppose $x\back z$ is defined and $y \leq x \back z$. Then $x(x\back z)$ is defined, since otherwise we would have: $x \in D_\ell$, $x \back z \leq \ell(x)$, and since $y \leq x \back z \leq \ell(x)$, then $xy$ would be undefined, a contradiction. Thus we get $xy \leq x (x \back z) \leq z$, by order preservation of multiplication.
Similarly one can prove that $xy \leq z$ if and only if  $z/y$ is defined and $x \leq z/y$.
\end{proof}
\begin{lemma}
If $I$ is a compatible ideal of an IRL $\alg C$, then $(\alg C', \ell_I, r_I)$ is an upper-compatible triple.
\end{lemma}
\begin{proof}
We show that $(\alg C', \ell_I, r_I)$ has the properties of an upper-compatible triple, recalling that $C' = C - I$. It is easy to check that $\alg C'$ is a partial IRL, in particular:
\begin{enumerate}
	\item Integrality is clearly satisfied;
	\item The three axioms of RLs are satisfied whenever they can be applied, in particular: 
	\begin{enumerate}
	\item with respect to the lattice operations, the  joins are always defined, $1$ is the largest element, and $x \land y$ is undefined iff there is no common lower bound of $x$ and $y$; 
	\item $1$ is the unit of the product, and $xy, (xy)z$ are defined iff they are not in $I$, iff $yz, x(yz)$ are not in $I$ and in such case $(xy)z = x/yz)$.
	\item residuation works by Lemma \ref{lemma:partialres}.
	\end{enumerate}
	\item  Since $\alg C$ is an IRL, multiplication is order preserving when defined;
	\item For the same reason, divisions are order-preserving in the numerator and order-reversing in the denominator whenever defined.
\end{enumerate}
In the rest of this proof we will write $\ell$ for $\ell_{I}$ and $r$ for $r_{I}$. We now check the properties in the definition of an upper compatible triple.
\begin{enumerate}
\item We first show that $\ell, r$ form a Galois connection whenever they are defined, i.e. that $\ell(y)$ is defined and $x \leq \ell(y)$ if and only if $r(x)$ is defined and $y \leq r(x)$.

If $r(x)$ is defined and $y \leq r(x)$, then $y \leq i/x$ for some $i \in I$, which by residuation is equivalent to $yx \leq i$. So $y$ is a left $I$-divisor and thus $\ell(y)$ is defined and $x \leq y \back i \leq \ell(y)$. Similarly, the converse holds.

\item We now prove that if $r(y')$ is defined, $x \leq r(y')$ and $y \leq y'$, then $r(y)$ is defined and $x \leq r(y)$. If $r(y')$ is defined then $y'$ is a right $I$-divisor, i.e. there is $z \in C'$ such that $zy' \in I$. Since $y \leq y'$, we have $zy \leq zy' \in I$ and since  $I$ is closed downwards we get $zy \in I$, i.e., 
 $y$ is a right $I$-divisor. Moreover, it follows from the definition of $r$ that it is order reversing, thus $r(y') \leq r(y)$. Since $x \leq r(y')$, we also have that $x \leq r(y)$. 

\item Given $x, y \in C'$, we now prove that $xy$ is undefined if and only if $y$ is in the domain of $r$ and $x \leq r(y)$, the proof of the other equivalence being similar. Notice that the product $xy$ is undefined in $C'$ if and only if it is an element of $I$. In such a case, we get $xy \leq i$ for some $i \in I$, thus $y$ is a right $I$-divisor and so it is in the domain of $r$. Moreover, $x \leq i/y \leq r(y)$. 

Conversely, if $y$ is in the domain of $r$ and $x \leq r(y)$, then $x \leq i/y$ for some $i \in I$. Thus $xy \in I$ and so $xy$ is undefined in $C'$.

\item We have that $x \back y$ is undefined in $C'$ iff it is an element of $I$, or equivalently, iff there is no $z \in C'$ with $xz \leq y$. Similarly, $y / x$ is undefined iff there is no $z$ with $zx \leq y$.

\item We need to show that, if $\ell(x)$ and $r(z)$ are defined: $x \back r(z)$ is defined iff $\ell (x) / z$ is defined and in such a case $x \back r(z)=\ell (x) / z$. 
Suppose first that $x \back r(z)$ is defined in $C'$. Then $(x\back r(z))z \leq x \back i \leq \ell(x)$, for some $i \in I$, which by residuation implies $x \back r(z) \leq \ell (x) / z$. Thus in particular $\ell(x)/z \in C'$. 

So we also get that $x(\ell(x) /z )z \leq x \ell(x) \leq j$ for some $j \in I$, which implies that $x(\ell(x) /z) \leq j/z \leq r(z)$, which again by residuation implies $\ell (x) / z \leq x \back r(z)$; thus the equality $x \back r(z)=\ell (x) / z$ is proved.  

Likewise, we can show that if $\ell (x) / z$ is defined in $C'$ then  $x \back r(z)$ is defined in $C'$ and $x \back r(z)=\ell (x) / z$. 
\item Suppose first that $\ell(x)$ is undefined and $r(z)$ is defined, then $r(z) = i/z$ for some $i \in I$. Thus $x \back (i/z) = (x \back i)/z \leq r(z)$, because necessarily $x \back i \in I$, and since $r(z) \leq x \back r(z)$ we have the equality. Similarly one can show the other case.
\item Now if $\ell(x)= \max\{x \back i: i \in I\}$ is defined in $C'$, then there exists $i \in I$ with $x \back i \in C'$. For every $z \in C-I$ we have $i <z$ and $x \back i \leq x \back z$, so $x \back z \in C'$, and thus it is defined in $C'$, and  
 $\ell(x) \leq x \back z$. The analogous fact for $r$ is proven similarly.
\item A meet $x\land y$ is undefined in $C'$ iff $x \land y \in I$ iff there is no $z \in C'$ with $z\leq x, y$. 
\item All other operations are defined. \end{enumerate}

\end{proof}

We say that $(\alg C', \ell_I, r_I)$ is the upper-compatible triple of the upper-compatible pair $(\alg C, I)$.

\begin{lemma}
Every upper-compatible triple $(\alg L, \ell, r)$ is the upper-compatible triple of the upper-compatible pair $(\alg C, J)$, where $J=\{0\}$ is a one-element set and $\alg C$ is an IRL with operations extending $\alg
 L \cup \{0\}$, with $0$ as the bottom element. 
\end{lemma}
\begin{proof}
Let $C = L \cup \{0\}$, with $0$ an idempotent element strictly below all elements of $L$. The operations of $\alg C$ are defined to extend the existing operations of $L$, and we further define:
\begin{align*}
x \land  y &=\begin{array}{ll}
0 & \mbox{ if } x \land y \mbox{ is undefined in } L \end{array}\\
x   y &=\begin{array}{ll}
0 & \mbox{ if } x = 0 \mbox{ or } y = 0 \mbox{ or } x \leq r(y)\end{array}\\
x \back  y &=\left\{\begin{array}{ll}
1 & \mbox{ if } x = 0\\
0 & \mbox{ if } x \not = 0, y = 0 \mbox{ and }\ell(x) \mbox{ is not defined} \\
\ell(x) &  \mbox{ if } x \not = 0, y = 0 \mbox{ and }\ell(x) \mbox{ is defined}  \end{array}\right.\\
y /  x &=\left\{\begin{array}{ll}
1 & \mbox{ if } x = 0\\
0 & \mbox{ if } x \not = 0, y = 0 \mbox{ and }r(x) \mbox{ is not defined} \\
r(x) &  \mbox{ if } x \not = 0, y = 0 \mbox{ and }r(x) \mbox{ is defined}  \end{array}\right.\\
\end{align*}

We set $J = \{0\}$ and show that $\alg C$ is an integral residuated lattice. The order defined clearly yields a lattice. Let us show that associativity holds, i.e. for any $x, y, z \in C$, $$x\cdot (y\cdot z) = (x \cdot y) \cdot z.$$
We distinguish the following cases.
\begin{itemize}
\item If any of $x, y, z$ is $0$, both sides of the equality are $0$ and thus associativity holds. The same holds if $x \leq r(y)$ and $ y \leq r(z)$.
\item Assume $x \leq r(y)$, and $r(z)$ is undefined or $y \not\leq r(z)$. Then $(x\cdot y) \cdot z = 0 \cdot  z = 0$. Moreover $yz$ is defined in $L$, thus $yz \leq y \leq \ell(x)$ (since there is a Galois connection between $\ell$ and $r$) hence $x \leq r(yz)$, so $x\cdot (y \cdot z) = 0$. 

Similarly we verify  the case where $y \leq r(z)$ and $x \not\leq r(y)$ or $r(y)$ is undefined.
\item Finally, assume that $r(y)$ is undefined or $x \not\leq r(y)$, and $r(z)$ is undefined or $y \not\leq r(z)$. Then the products $xy, yz$ are defined in $L$, and then we get that $(x  y)  z  = 0$ if $xy \leq r(z)$, and $(xy)z \in L$ otherwise. Similarly, $x (yz) = 0$ if $yz \leq \ell(x)$, and $x(yz) \in L$ otherwise. 

The claim is proved by showing that $xy \leq r(z)$ if and only if $yz \leq \ell(x)$. 
Indeed, suppose $r(z)$ is defined and $xy \leq r(z)$, by Lemma \ref{lemma:partialres} we get that $x \back r(z)$ is defined and $y \leq x \back r(z)$. Then if $\ell(x)$ is undefined, by Property \ref{propassoc2} we get $x \back r(z) = r(z)$, thus $y \leq r(z)$, a contradiction. Then also $\ell(x)$ is defined, thus by Property \ref{propassoc} $\ell(x)/z$ is defined and $y \leq x \back r(z) = \ell(x) /z$. Since $\ell(x)/z$ is defined and $y \leq \ell(x) /z$, by Lemma \ref{lemma:partialres} we obtain $yz \leq \ell(x)$ since $yz$ is defined. Similarly one can show the right-to-left direction.

\end{itemize} 
We now show residuation: $$xy \leq  z \mbox{ iff } y \leq  x \back z \mbox{ iff } x \leq  z/ y$$
\begin{itemize}
\item If $x= 0$ or $y = 0$ the claim is easily shown. Suppose now $z  = 0$; we need to show $xy \leq 0$ iff $y\leq x \back 0$. If $\ell(x)$ is defined then $x\back 0 = \ell(x)$, and we know that $xy \leq 0$ iff $ y \leq \ell(x)$. If $\ell(x)$ is not defined then $x\back 0 =0$ and and $xy$ is defined to be $0$, so both inequalities hold.
\item Let $x,y,z \neq 0$, $r(y)$ defined, and $x \leq r(y)$ (or equivalently $\ell(x)$ defined and $y \leq \ell(x)$). Then $xy$ is not defined in $L$, thus the first inequality becomes $0 \leq z$ and is true. Since $\ell(x)$ is defined, by Property \ref{prop-orderell} $x \back z$ is also defined and $y \leq \ell(x) \leq x \back z$, so 
the second inequality is also true. The proof that $ x \leq  z/ y$ holds is similar.
\item  Let $x,y,z \neq 0$, $r(y)$ undefined or $x \not\leq r(y)$; then $xy$ is defined in $L$. Residuation follows from Lemma~\ref{lemma:partialres}.
\end{itemize}
Notice that if $c$ is a left $J$-divisor, $\ell_{J}(c) = \max \{c \back  i: i \in J\} = c\back  0 = \ell(c)$, and if $d$ is a right $J$-divisor $r_{J}(d) = \max \{i /  d: i \in J\} = 0 \back  d = r(d)$. Thus we have shown that $(\alg L, \ell, r)$ is the upper-compatible triple of the upper-compatible pair $(\alg C, J)$. 
\end{proof}

\begin{corollary}
If $(\alg C, I)$ is an upper-compatible pair, so is $((\alg C -I) \cup \{0\}, \{0\})$.
\end{corollary}

\subsection{The gluing over a filter-ideal pair}
We are now ready to introduce the gluing construction over a congruence filter and an ideal, which is depicted in Figure \ref{fig:gluing1}. To ease the notation, we write the pair of the filter $F$ and the ideal $I$ as $P$: $$P := (F, I)$$ 
We can then define the gluing of $\alg B$ and $\alg C$ over the pair $P = (F, I)$, or $(F-I)$-gluing of $\alg B$ and $\alg C$, as the structure
 $\alg B  \oplus_{P} \alg C$ where the operations extend the ones of $\alg B$ and $\alg C$ as follows:
\begin{figure}
\begin{center}
\begin{tikzpicture}
\footnotesize{
 \draw  [dashed](0,0.5) ellipse (0.3 and 0.4);
 \draw  (0,1.5) ellipse (0.3 and 0.4);
  \draw  (0,-0.5) ellipse (0.3 and 0.4);
  \fill (0,1.9) circle (0.05);
 \node at (3,1.5) {$F$}; 
  \node at (0,1.5) {$F$}; 
   \node at (6,1.9) {$F$}; 
 \node at (3,-0.5) {$I$}; 
  \node at (0,-0.5) {$I$}; 
   \node at (6,-1) {$I$};
 \node at (0.15,2.1) {$1$}; 
  \node at (0,-1.7) {$\alg B$};
   \node at (3,0.6) {$C^{-}$}; 
  \node at (0,0.6) {$B^{-}$}; 
     \node at (6,1) {$C^{-}$}; 
  \node at (6,0) {$B^{-}$}; 
  \node at (1.5,0.5) {$\oplus_{P}$};
  \draw [densely dotted] (3,0.5) ellipse (0.3 and 0.4);
    \draw  (3,-0.5) ellipse (0.3 and 0.4);
  \draw  (3,1.5) ellipse (0.3 and 0.4);
  \fill (3,1.9) circle (0.05);
 \node at (3.15,2.1) {$1$}; 
  \node at (3,-1.7) {$\alg C$};
   \node at (4.5,0.5) {${=}$};
    \draw  [dashed](6,-0.05) ellipse (0.3 and 0.4); 
    \draw [densely dotted] (6,0.95) ellipse (0.3 and 0.4);
 \node at (6.15,2.5) {$1$}; 
 \draw  (6,1.9) ellipse (0.3 and 0.4);
 \draw  (6,-1) ellipse (0.3 and 0.4); 
  \node at (6,-1.8) {$\alg B \oplus_{P} \alg C$};
\fill (6,2.3) circle (0.05);  
   }
 \end{tikzpicture} \caption{The gluing $\alg B \oplus_{P} \alg C$ of the algebras $\alg B$ and $\alg C$ over the pair $P = (F, I)$.} \label{fig:gluing1}
   \end{center}  
 \end{figure}
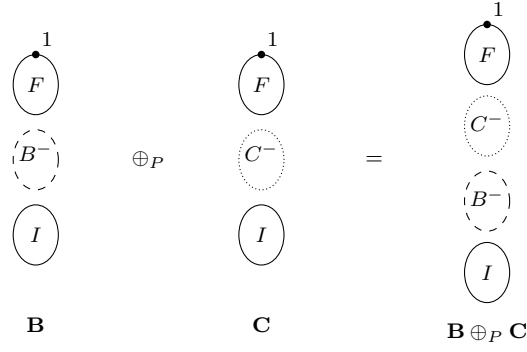

\begin{align*}
x \cdot  y &=\left\{\begin{array}{ll}
x \cdot y & \mbox{ if } x, y \in  B, \mbox{ or } x, y \in  C\\
\sigma_{F}(x) & \mbox{ if } y \in C^{-}, x \in B^{-}\\
\sigma_{F}(y) & \mbox{ if } x \in C^{-}, y \in B^{-}\end{array}\right.\\
x\back \, y &=\left\{\begin{array}{ll} 
x \back y & \mbox{ if } x, y \in  B, \mbox{ or } x, y \in  C\\
\gamma_{F}(y) & \mbox{ if } y \in B^{-} \mbox { and } x \in C^{-} \mbox{ is not a left $I$-divisor} \\
\ell_{I}(x) & \mbox{ if } y \in B^{-} \mbox { and } x \in C^{-} \mbox{ is a left $I$-divisor}\\
1 & \mbox{ if } x \in B^{-}, y \in C^{-} \end{array}\right.\\
x/ \, y &=\left\{\begin{array}{ll} 
x / y & \mbox{ if } x, y \in  B, \mbox{ or } x, y \in  C\\
\gamma_{F}(x) & \mbox{ if } x \in B^{-} \mbox { and } y \in C^{-} \mbox{ is not a right $I$-divisor} \\
r_{I}(y) & \mbox{ if } x \in B^{-} \mbox { and } y \in C^{-} \mbox{ is a right $I$-divisor} \\
1 & \mbox{ if } x \in C^{-}, y \in B^{-} \end{array}\right.\\
x \land  y &=\left\{\begin{array}{ll}
x \land y & \mbox{ if } x, y \in  B, \mbox{ or } x, y \in  C \mbox{ with } x \land y \not\in I\\
\top_{B} & \mbox{ if } x, y \in  C^{-},  x \land y \in I\\
x & \mbox{ if } x \in B^{-}, y \in C^{-}\\
y & \mbox{ if } y \in B^{-}, x \in C^{-}\end{array}\right.\\
x \lor  y &=\left\{\begin{array}{ll}
x \lor y & \mbox{ if } x, y \in  C, \mbox{ or } x, y \in  B \mbox{ with } x \lor y \not\in F\\
\bot_{C} & \mbox{ if } x, y \in  B^{-},  x \lor y \in F\\
y & \mbox{ if } x \in B^{-}, y \in C^{-}\\
x & \mbox{ if } y \in B^{-}, x \in C^{-}\end{array}\right.
\end{align*}

\begin{theorem}\label{prop:quadruplegluing}
If $(\alg B, F, \alg C, I)$ is a compatible quadruple, then $\alg B \oplus _{P} \alg C$ is the gluing of $\alg B$ and $\alg C$ over $F$ and $I$. 
\end{theorem}

\begin{proof}
We show that $\alg B \oplus_P  \alg C$ is an IRL. 
The fact that $\alg B \oplus_P  \alg C$ has an underlying lattice structure is guaranteed by the order properties of the compatible quadruple. In particular, it follows that $F$ is strictly above all other elements of $B$ and $C$ implies that if $x, y \in B-F$ (or $x, y \in C-F$) and $x \lor y \in F$, then $F$ has a bottom element $\bot_{F}$ and $x \lor y = \bot_{F}$. Also, if there are $x, y \in B-F$ with $x \lor_{B} y \in F$, then $x \lor y = \bot_{C}$. This is not in conflict with the definition of the operations, since we can show that $\sigma_{F}(x)= \bot_{F} \cdot x = x  \cdot \bot_{F}$ and $\gamma_{F}(x) = \bot_{F} \back x = x /\bot_{F}$ in $\alg B$.
First, it is easy to see that $\bot_{F}  \cdot x \in [x]_{F}$; indeed, $(\bot_{F}  \cdot x) \back x = 1$ and $\bot_{F} \leq x \back (\bot_{F} \cdot  x)$, thus both $(\bot_{F}  \cdot x) \back x$ and $x \back( \bot_{F} \cdot  x)$ are in $F$. 
Moreover, $(\bot_{F} \cdot  x) \back \sigma_{F}(x) = \bot_{F} \back (x \back \sigma_{F}(x)) = 1$ since $x \back \sigma_{F}(x) \in F$ and thus   $\bot_{F} \cdot  x \leq \sigma_{F}(x)$. Therefore, $\sigma_{F}(x)= \bot_{F} \cdot x$ and the proof for $x  \cdot \bot_{F}$ is analogous.

We now show that $\gamma_{F}(x) = \bot_{F} \back x$; the proof for $x /\bot_{F}$ is similar. Since $x \leq \bot_{F} \back x$, we get $x \back (\bot _{F}\back x)=1 \in F$, and we also have $(\bot_{F} \back x) \back x \geq \bot_{F}  \in F$; hence $ \bot _{F}\back x \in [x]_{F}$.
Moreover, $\gamma_{F}(x) = \max[x]_{F} \leq \bot_{F} \back x$, or equivalently, $\bot_{F} \gamma_{F}(x) \leq x$, since $\bot_{F} \gamma_{F}(x) = \min[\gamma(x)]_{F} = \min[x]_{F} = \sigma_{F}(x) \leq x$.

Similarly, since $I$ is strictly below all other elements of $B$ and $C$,  if $x, y \in B-I$ (or $x, y \in C-I$) and $x \land y \in I$, then $I$ has a top element $\top_{I}$ and $x \land y = \top_{I}$. Thus given $x, y \in C-I$ with $x \land_{C} y \in I$, the meet is redefined as $x \land y = \top_{B}$. This is not in conflict with the definition of the operations, due to Lemma \ref{bounded-lemma} (\ref{condC-}), and Lemma \ref{lemma:filtercompatible2}.

Also, it is clear that $1$ is both the monoidal unit and the top element of the lattice. 

To prove associativity, we need to show that for every $x,y,z \in B \oplus_P  C$, $(x   y)   z = x   (y  z)$. We distinguish the following cases.
\begin{itemize}
\item Let $x \in F, y \in C^{-}, z \in B^{-}$. Then $(x   y)   z = \sigma_F(z)$, since $xy \in C$ from Lemma~\ref{bounded-lemma}(\ref{condC-}).  
Now, $x   (y   z) = x \sigma_F(z) = \sigma_F(z)$, given 
Lemma~\ref{bounded-lemma}(\ref{conuclei-abs-b}). Similarly we can show the cases where: $x \in B^{-}, y \in C^{-}, z \in F$; $x \in F, y \in B^{-}, z \in C^{-}$; $x \in C^{-}, y \in F, z \in B^{-}$, $x \in C^{-}, y \in B^{-}, z \in F$; $x \in B^{-}, y \in F, z \in C^{-}$.
\item Let $x, y \in C^{-}, z \in B^{-}$. We have that $(x  y)   z = \sigma_F(z)$,  if either $xy \in C^{-}$ (by definition) or if $xy \in I$ (by the compatibility condition \ref{compat-cond-ideal} for the quadruple). On the other hand, $x  (y   z) = x \sigma_F(z) = \sigma_F(\sigma_F(z))= \sigma_F(z)$.
The proof is analogous for the case: $x \in B^{-}, y,z \in C^{-}$.
\item If $x \in C^{-}, y,z \in B^{-}$, then $(x   y)   z = \sigma_F(y) z = \sigma_F(yz)$, given that $\sigma_F$ is a strong conucleus. Also, $x   (y  z) = \sigma_F(yz)$, if either $ yz \in B^{-}$  (by definition) or $yz \in I$ (by Lemma \ref{lemma:filtercompatible2}).  We get a similar proof for the cases: $x, y \in B^{-}, z \in C^{-}$; $x \in B^{-}, y \in C^{-}, z \in I$; $x \in I, y \in C^{-}, z \in B^{-}$; $x \in I, y \in B^{-}, z \in C^{-}$; $x \in B^{-}, y \in I, z \in C^{-}$; $x, z \in B^{-}, y \in C^{-}$; $x, z \in C^{-}, y \in B^{-}$. 
\item Since both $\alg B$ and $\alg C$ are subalgebras with respect to multiplication, the remaining cases hold automatically.
\end{itemize}
We now prove that for all $x, y, z$, $$ x \cdot  y \leq z \mbox{ iff } x \leq z / y \mbox{ iff } y \leq x \back  z$$
We have the following cases:
\begin{itemize}
\item Let $x \in F, y \in C^{-}, z \in B^{-}$. Then it never happens that $x \cdot  y \leq z$, by Lemma \ref{bounded-lemma}(\ref{condC-}). The other inequalities are also false by definition and order preservation. An analogous case is given by $x \in C^{-}, y \in F, x \in B^{-}$.
\item Let $x \in F, y \in B^{-}, z \in C^{-}$. Then all three inequalities are always true, given the definition of the operations and order preservation. Similar cases are given by: $x \in C^{-}, y \in B^{-}, z \in F$; $x, z \in C^{-}, y \in B^{-}$; $x \in B^{-}, y \in F, z \in C^{-}$; $x \in B^{-}, y \in C^{-}, z \in F$; $ x \in B^{-}, y,z \in C^{-}$; $x, y \in B^{-}, z \in C^{-}$; $x \in B^{-}, y \in I, z \in C^{-}$; $x \in I, y \in C^{-}, z\in B^{-}$; $x \in I, y \in B^{-}, z \in C^{-}$; $x \in C^{-}, y \in I, z \in B^{-}$. 
\item Let $x, y \in C^{-}, z \in B^{-}$.
We distinguish two cases, based on whether $xy \in I$ or not. If  $xy = i$ for some $i \in I$, then all inequalities hold. Indeed $xy = i\leq z$; moreover $xy = i$ implies $y \leq x \back i \leq \ell_{I}(x) = x \back  z$ and similarly $x \leq i/y \leq r_{I}(y) = z/ y$.
If $xy \in C^{-}$, none of the inequalities hold. Indeed, $xy \not\leq z$ by definition of the order. Moreover, if $y \leq x \back  z$ the only possibility by definition is that $x \back  z = \ell_{I}(x)$, but $x \ell_{I}(x) \in I$ thus we would have $xy \in I$, a contradiction. Similarly it cannot be that $x \leq z / y$. 

\item Let $x \in C^{-}, y,z \in B^{-}$. To show that $ x  y = \sigma_F(y) \leq z $ iff $x \leq z / y$ it suffices to note that, equivalently, $\sigma_F(y) \leq z$ iff $z/y \in F$. Indeed since $\sigma_F(y)/y \in F$, we have that $\sigma_F(y) \leq z$ implies $z/y \in F$. Conversely, if $z/y \in F$ then there is $f \in F$ such that $f \leq z/y$, thus $fy \leq z$, and so $\sigma_F(y) = \sigma_F(fy) \leq fy \leq z$ (where in the first equality we used Lemma~\ref{bounded-lemma}(\ref{conuclei-abs-b})).

We now show that $x \cdot  y = \sigma_F(y) \leq z$ iff $y \leq x\back  z$. If $\sigma_F(y) \leq z$, then since $\sigma_F$ and $\gamma_F$ form a Galois connection we get $y\leq  \gamma_F(z) \leq x\back  z$, where the second inequality holds because $\gamma_F(z) = x\back z$ or $x\back z\in C^{-}$.
Conversely, assume $y \leq x\back  z$. If $x\back  z = \ell_{I}(x)$, then by definition $x$ is a left $I$-divisor thus there is a $c \in C^{-}$ such that $xc \in I$, thus by the compatibility condition (\ref{compat-cond-ideal}) for the quadruple $\sigma_F(y) \in I$ thus $\sigma_F(y) \leq z$. Otherwise, if $x\back z  = \gamma_F(z)$ then $y \leq \gamma_F(z)$, so by the Galois connection we have that $\sigma_F(y) \leq z$.

An analogous case is given by $x, z \in B^{-}, y \in C^{-}$.
\item Let $x \in C^{-}, y \in B^{-}, z \in I$. The fact that $x y \leq z$ iff $x \leq z / y$ can be shown, as in the previous case, using the fact that  $\sigma_F(y) \leq z$ if and only if $z/y \in F$.

We now show that $xy = \sigma_F(y) \leq z$ iff $y \leq x \back z$. Since $\alg C$ is a subalgebra, either $x \back z \in C^{-}$ or $x \back z \in I$. If $x \back z \in C^{-}$ then clearly $y \leq x \back z$ and by the compatibility condition \ref{compat-cond-ideal} for the quadruple, we get $ \sigma_F(y) = xy \leq x (x\back z) \leq z$.

Otherwise, we have $x \back z \in I$, so $y\not\leq x \back z$, since $y \in B^{-}$. Moreover, $x \back z = \gamma_{F}(z)$ from Lemma \ref{lemma:filtercompatible2}. Thus since $y \not\leq \gamma_{F}(z)$, we get that $\sigma_{F}(y) \not\leq z$, given that the two operators are a residuated pair.

Similarly we prove residuation for the case: $x \in B^{-}, y \in C^{-}, z \in I$.
\item Since both $\alg B$ and $\alg C$ are subalgebras for the divisions the other cases do not need to be checked.
\end{itemize}
Thus, $\alg B \oplus_P \alg C$ is an integral residuated lattice.\end{proof}
Notice that in the case where $\II$ is empty, we get the proof of Theorem \ref{prop:gluingfilter}. Indeed, if $I$ is empty, no elements of $C-F$ are $I$-divisors and thus we obtain exactly the hypothesis of Theorem \ref{prop:gluingfilter}.
\subsection{A gluing of partial algebras}
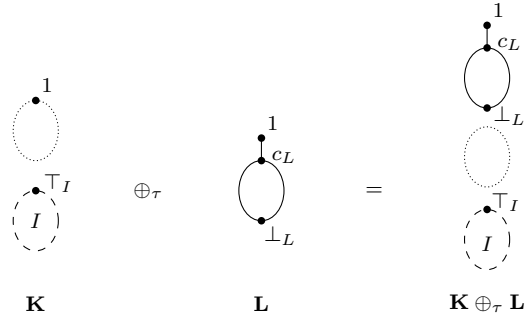
\begin{figure}
\begin{center}
\begin{tikzpicture}
\footnotesize{
 \draw  [dashed](0,0.1) ellipse (0.3 and 0.4);
 \draw  [densely dotted](0,1.3) ellipse (0.3 and 0.4);
  \fill (0,1.7) circle (0.05);
   \fill (0,0.5) circle (0.05);
  \node at (0,0.1) {$I$}; 
   \node at (6,-0.2) {$I$}; 
 \node at (0.15,1.9) {$1$}; 
  \node at (0,-1) {$\alg K$};
  \node at (0.3,0.6) {$\top_I$}; 
  \node at (6.25,0.35) {$\top_I$}; 
  \node at (1.5,0.5) {$\oplus_{\tau}$};
  \draw  (3,0.5) ellipse (0.3 and 0.4);
  \fill (3,0.9) circle (0.05);
    \fill (3,0.1) circle (0.05);
    \fill (6,1.6) circle (0.05);
    \fill (6,0.25) circle (0.05);
\draw (3,0.9)--(3,1.2);
\fill (3,1.2) circle (0.05);
 \node at (3.3,0.95) {$c_L$}; 
  \node at (3.15,1.4) {$1$}; 
  \node at (3.25,-0.1) {$\bot_L$}; 
  \node at (6.3,1.5) {$\bot_L$}; 
  \node at (3,-1) {$\alg L$};
   \node at (4.5,0.5) {${=}$};
    \draw  [dashed](6,-0.15) ellipse (0.3 and 0.4); 
    \draw [densely dotted] (6,0.95) ellipse (0.3 and 0.4);
 \node at (6.3,2.45) {$c_L$};
  \node at (6.15,2.9) {$1$}; 
 \draw (6,2.4)--(6,2.7);
\fill (6,2.7) circle (0.05);
 \draw  (6,2) ellipse (0.3 and 0.4);
  \node at (6,-1) {$\alg K \oplus_{\tau}\alg L$};
\fill (6,2.4) circle (0.05);  
   }
 \end{tikzpicture} \caption{The partial gluing $\alg K \oplus_{\tau}\alg L$, in the case where $\alg L$ has a coatom $c_L$.} \label{fig:gluing2}
   \end{center}  
\end{figure}
In this section, we obtain a different construction that glues together two structures that intersect at the top $1$ and keeps the same order relation, but where some of the divisions are redefined. The underlying idea is to forget the filter $F$ from the previous construction. See Figure \ref{fig:gluing2} for a pictorial intuition.

We start from a lower-compatible triple $(\alg K, \sigma, \gamma)$ and an upper-compatible triple $(\alg L, \ell, r)$.  Recall that in upper-compatible triples already some divisions are not defined. Here we will allow other divisions not to be defined. Precisely, we shall say that $x \back y$ is \emph{strongly undefined} (in order to distinguish this case in the definition of the operations) if all elements $z$ in the interval $[y, 1) = \{z \in L: y \leq z < 1\}$ are such that $xz \leq y$ and there is no coatom. Similarly, $y /x$ is strongly undefined if all elements $z$ the interval $[y, 1)$ are such that $zx \leq y$ and $\alg L$ has no coatom.
We assume: 
\begin{enumerate}
\item[(A1)] $\alg K$ has an ideal $I \subseteq K$ with an idempotent top element $\top_{I}$ such that $\sigma(\top_{I}) = \top_{I}$.
\item[(A2)] If there are undefined products in $L$, then $\sigma(x) = \top_{I}$ for all $x \in K-I$ and $\sigma(\top_{I}y)= \sigma(y\top_{I}) = \sigma(y)$ if $y \in I$.
\item[(A3)]  If there exists $x, y \in K - \{1\}$ such that $x \lor y = 1$, then $L$ has a bottom element $\bot_L$.
\item[(A4)]  If $\alg L$ has undefined meets, $\alg K$ has a splitting coatom $c_K$.
\end{enumerate}
Moreover, we assume that $K \cap L = \{1\}$, we set $\tau := (\sigma, \gamma, \ell, r)$ and we define the \emph{partial gluing} $\alg K \oplus_{\tau}  \alg L$ to be the structure where the operations extend the ones of $\alg K$ and $\alg L$ in the following way, here $D_\ell$ and $D_r$ denote the domains of $\ell$ and $r$, respectively:
\begin{align*}
x y &=\left\{\begin{array}{ll}
x  y & \mbox{ if } x, y \in  K \mbox{ or } x, y \in  L \mbox{ and $xy$ is defined}\\
\sigma(x) & \mbox{ if } y \in L - \{1\}, x \in K - \{1\}\\
\sigma(y) & \mbox{ if } x \in L - \{1\}, y \in K - \{1\}\\
\top_{I} & \mbox{ if } x, y \in L \mbox{ and $xy$ is undefined}\\\end{array}\right.\\
x\back\, y &=\left\{\begin{array}{ll} 
x \back y & \mbox{ if } x, y \in  K \mbox{ or } x, y \in  L, \mbox{ and $x \back y$ is defined }\\
c_{L} & \mbox{ if } x, y \in  K, L \mbox{ has a coatom } c_{L} \mbox{ and $x \back y$ is undefined}\\
\ell(x) & \mbox{ if } x \in L - \{1\}, y \in K - \{I \cup 1\}  \mbox{ and } x \in D_{\ell}, \mbox{ or } x, y \in L \mbox{ and $x \back y$ undefined} \\
\gamma(y) & \mbox{ if } x \in L - \{1\}, y \in I \mbox{ or }  (y \in K - \{1\} \mbox{ and } x \not\in D_{\ell})\\
1 & \mbox{ if } x \in K - \{1\}, y \in L\\
\mbox{ undefined } & \mbox{ if } x,y \in L, \mbox{ and $x \back y$ strongly undefined} \end{array}\right.\\
y/\, x &=\left\{\begin{array}{ll} 
y / x & \mbox{ if } x, y \in  K \mbox{ or } x, y \in  L, \mbox{ and $y / x$ is defined }\\
 c_{L} & \mbox{ if } x, y \in  K, L \mbox{ has a coatom } c_{L} \mbox{ and $y/x$ is undefined}\\
r(x) & \mbox{ if } x \in L - \{1\}, y \in K - \{I \cup 1\}  \mbox{ and } x \in D_{r} \mbox{ or } x, y \in L \mbox{ and $y /x$ undefined} \\
\gamma(y) & \mbox{ if } x \in L - \{1\}, y \in I \mbox{ or } (\in K - \{1\} \mbox{ and } x \not\in D_{r}) \\
1 & \mbox{ if } x \in K - \{1\}, y \in L\\
\mbox{ undefined } & \mbox{ if } x,y \in L, \mbox{ and $y /x$ strongly undefined} \end{array}\right.\\
x \land y &=\left\{\begin{array}{ll}
x \land y & \mbox{ if } x, y \in  K, \mbox{ or } x, y \in  L\\
c_K & \mbox{ if } x, y \in L \mbox{ and } x\land y \mbox{ undefined}\\
y & \mbox{ if } y \in K - \{1\}, x \in L\\
x & \mbox{ if } x \in K - \{1\}, y \in L\end{array}\right.\\
x \lor y &=\left\{\begin{array}{ll}
x \lor y & \mbox{ if } x, y \in  K, \mbox{ or } x, y \in  L\\
x & \mbox{ if } y \in K -\{1\}, x \in L\\
y & \mbox{ if } x \in K -\{1\}, y \in L\\
\bot_L & \mbox{ if } x \lor y = 1 \mbox{ in } K.\end{array}\right.
\end{align*}

\begin{theorem}\label{prop:partialgluing}
$\alg K \oplus_{\tau}  \alg L$ is a partial IRL, that is total if $L$ has a coatom.
\end{theorem}
\begin{proof}
Notice first that all operations are defined except possibly the divisions.
It is easy to see that the operations $\land, \lor$ define a lattice order, with $1$ being the top. 
Let us now prove associativity of multiplication.
\begin{itemize}
\item Suppose $x,y,z \in L$. Then by definition:
  \begin{align*}
  (x  y) z &=\left\{\begin{array}{ll}
(xy)z & \mbox{ if } xy \mbox{ and } (xy)z \mbox{ are defined in } L\\
\sigma(\top_{I} ) = \top_{I} & \mbox{ otherwise }\end{array}\right.\\
  x  (y z) &=\left\{\begin{array}{ll}
x(yz) & \mbox{ if } yz \mbox{ and } x(yz) \mbox{ are defined in } L\\
\sigma(\top_{I} ) = \top_{I} & \mbox{ otherwise }\end{array}\right.
\end{align*}
 We will show that $xy$ and $(xy)z$ are defined if and only if $yz$ and $x(yz)$ are defined, or equivalently, $yz$ or $x(yz)$ is undefined iff $xy$ or $(xy)z$ is undefined. Notice that when $(xy)z$ and $x(yz)$ are defined they coincide, since in upper compatible triples the IRL axioms hold whenever the operations involved are defined. 

We show the left-to-right direction first: we assume that $yz$ is undefined or $x(yz)$ is undefined; we also assume that $xy$ is defined. If $yz$ is undefined, then $z \in D_{r}$ and $y \leq r(z)$. Since $xy \leq y$, we get $xy \leq r(z)$, so $(xy)z$ is undefined. If instead $yz$ is defined and $x(yz)$ is undefined, we have that $yz \in D_{r}$ and $x\leq r(yz)$. Then by Property \ref{propUC1} of an upper compatible triple, $\ell(x)$ is defined and $yz \leq \ell(x)$. Suppose $r(z)$ is defined, thus using Lemma \ref{lemma:partialres} and Property \ref{propassoc}, $y \leq \ell(x)/z = x \back r(z)$ and then $xy \leq r(z)$; thus $(xy)z$ is undefined. We also show that $r(z)$ is necessarily defined, indeed if $r(z)$ is undefined, $\ell(x) /z = \ell(x)$, and $y \leq \ell(x)$ implies $x \leq r(y)$, thus $xy$ undefined. 

The other direction is proved in a similar way.

\item For $x, y \in L$ and $z \in K-I$, we show that $x  (y z) = \sigma(z)$. If $x \not \leq r(y)$, then $(x  y) z = \sigma(z)$ holds by definition. Otherwise we get $\top_{I} z = \top_{I}$, since $\top_{I}$ is idempotent; so there is at least an undefined product in $L$, thus by  (A2) we get $\sigma(z) = \top_{I}$.
Similarly we can show the case: $y, z \in L$, $x \in K-I$.
\item For $x, y \in L$ and $z \in I$, we show that $x  (y z) = \sigma(z)$. If $x \not \leq r(y)$, then $(xy) z = \sigma(z)$ by definition. Otherwise we get $\top_{i} z = \sigma(z)$ since $\sigma(z) = \sigma(\top_{I}z)$ implies $\sigma(z) \leq \top_{I}z$ and $\top_{I}z \leq \sigma(z)$.
Similarly we prove the case: $y, z \in L$ and $x \in I$.
\item For $x, z \in L, y \in K$, associativity follows directly from the definition of multiplication and the idempotency of $\sigma$.
\item For $x \in L, y,z \in K$, associativity follows from the definition of multiplication and the strong conuclear property. Similar cases are: $x, z \in K, y \in L$; $x, y \in K, z \in L$.
\item In the remaining cases all elements belong to $K$.
\end{itemize}
It easily follows that the product is a monoidal operation with unit $1$. 
We now show that residuation holds (when the divisions are defined): $$x  y \leq z \mbox{ iff } y \leq x \back z \mbox{ iff } x \leq z/ y$$
\begin{itemize}
\item For $x, y, z \in L$, we distinguish two cases, depending whether $xy$ is defined in $L$ or not. If $xy$ is not defined in $L$, then we get that all three inequalities hold since they respectively become: $\top_{I} \leq z, y \leq \ell(x), x \leq r(y)$; here  we used Property \ref{propUC3} of upper-compatible triples.

If $xy$ is defined in $L$, then we get $xy \leq z$ if and only if $x \back z$ is defined and $y \leq x\back z$, by Lemma \ref{lemma:partialres}. Similarly, this is equivalent to  $z/y$ being defined and $x \leq z/y$. 
\item For $x,y \in L, z \in K-\{I \cup 1\}$, we distinguish two cases, based on whether $xy$ is defined in $L$ or not. If $xy$ is defined in $L$, we have $xy \not\leq z$ by the definition of the order; also $x \not \leq r(y)$ and thus $y \not\leq \ell(x)$, by Property \ref{propUC1} (Galois connection). Moreover, $x \back z$ is either $\ell(x)$ or $\gamma(z)$, and in either case $y \not\leq x\back z$, since $\gamma(z) \in K-1$ and $y \not\leq \ell(x)$.

If $xy$ undefined in $L$, then $xy$ is equal to the top element of $I$, and $\top_{I} \leq z$. Moreover, $y \leq \ell(x) = x \back z$, and $x \leq r(y) = z/y$, using again Property \ref{propUC1},\ref{propUC3}.
\item For $x,y \in L, z \in I$, we distinguish two cases, whether $x \leq r(y)$ or not. If $x \not\leq r(y)$, the proof is the same as in the previous case. If $x \leq r(y)$ we distinguish whether $z = \top_I$ or not. If $z < \top_I$, then $xy = \top_{I} \not\leq z$, $y \not\leq \gamma(z) = x \back z$ and $x \not\leq \gamma(z) = z/y$. If $z = \top_I$, then $xy = \top_I$ implies $x \leq \top_I /y = r(y)$  iff $y \leq \ell(x) = x \back \top_I$.
\item If $x, z \in L, y \in K-\{1\}$, it follows directly from the definition of the operations that all inequalities hold, whenever the divisions are defined. Similarly for the cases: $x \in L, y \in I, z \in K-\{I\}$; $x \in K, y,z \in L$; $x, y \in K, z \in I$.
\item For $x \in L, y,z, \in K-\{I\}$, it follows directly from the definition of the operations that $xy = \sigma(y) \leq z$ iff $x \leq z/y$ whenever the division is defined. Indeed $z/y$ is either $z/_{K}y$ iff $\sigma(y)\not\leq z$ and it is either undefined or $c_{L}$ if $\sigma(y) \leq z$. Now we show that $xy = \sigma(y) \leq z$ iff $y \leq x \back z$. If $x \in D_{\ell}$ then there are undefined products in $L$, thus $\sigma(y) = \top_{I} \leq z$ and $x \back z = \ell(x)$ thus $y \leq x \back z$. Otherwise, if $x \not\in D_{\ell}$, $\sigma(y) \leq z$ iff $y \leq \gamma(z) = x \back z$ since $\sigma, \gamma$ form a residuated pair.
The following cases are proven similarly: $x \in L, y \in K, z \in I$; $x, z \in K-\{I\}, y \in L$; $x \in K-\{I\}, y\in L, z \in I$.
\item For $x, y, z \in K$, we show that $xy \leq z $ iff $y \leq x \back z$, as the proof the equivalence $xy \leq z \Leftrightarrow x \leq z/y$ is analogous. If $x \back_{K} z $ is defined then residuation holds since $K$ is a partial IRL. If $x \back_{K} z $ is undefined, then $\sigma(x) \leq z$, and $xy \leq \sigma(x) \leq z$ holds. Moreover $x\back z $ is either still undefined or if there is a coatom $x \back z = c_{L}$, thus $y \leq c_{L}$.
\end{itemize}
We now show that multiplication is order preserving: if $x \leq y$, then $xz \leq yz$ and $zx \leq zy$ (notice again that all products are defined in the gluing).
\begin{itemize}
	\item The cases where $x, y \in L$ and $z \in K-\{1\}$ follow directly from the definition of the operations.
	\item The cases where $x, y \in K -\{1\}$ follow from the order preservation of $\sigma$. 
	\item If $x, z \in K-\{1\}, y \in L$, then we get $xz, zx \leq \sigma(z)$, which is a property of $\sigma$ in a lower compatible triple.
	\item Let now $x \in K -\{1\}$, $y, z \in L$. If $yz$ (equiv. $zy$) is defined in $L$, clearly $xz \leq yz$ (equiv. $zx \leq zy$). If $yz$ (or $zy$) is undefined in $L$, the inequalities becomes $\sigma(x) \leq \top_I$, which holds by condition (A2).
	\item Finally, let $x, y, z \in L$. We show order preservation of right multiplication, the proof for left multiplication being analogous. If $xz, yz$ are defined in $L$, then order preservation holds since $\alg L$ is a partial IRL. If $yz$ is undefined, by the definition of an upper compatible triple, $z \leq \ell(y)$, and since $x \leq y$, and the domain of $\ell$ is closed downwards, by property (2) of the definition we also get that $z \leq \ell(x)$ and thus $xz$ is undefined in $L$. Therefore, $xz = \top_I = yz$. Suppose now that $yz$ is defined and $xz$ is undefined, then the inequality becomes $\top_I \leq yz \in L$, which holds by definition of the order in the gluing. 
\end{itemize}
The fact that (when defined) divisions are order preserving in the numerator and order reversing in the denominator follows from residuation and the order preservation of multiplication (which is always defined in $\alg K \oplus_{\tau}  \alg L$).
Thus, $\alg K \oplus_{\tau}  \alg L$ is a partial IRL.
We note that in the case where $L$ has a coatom no operation is undefined.
\end{proof}

If we take a compatible quadruple $(\alg B, F, \alg C, I)$ where $I$ has a top element satisfying assumptions $(A1), (A2)$, we can then consider the lower compatible triple $(\alg B', \sigma_{F}, \gamma_{F})$ (where $B' = B-F$) and the upper compatible triple $(\alg C', \ell_{I}, r_{I})$ (where $C' = C-(F \cup I)$) and construct the partial gluing $\alg B' \oplus_{\tau} \alg C'$, where $\tau = (\sigma_{F}, \gamma_{F},\ell_{I}, r_{I})$.

More generally, starting from any compatible quadruple $(\alg B, F, \alg C, I)$, we can always construct a partial gluing. Indeed either there are no $I$-divisors in $\alg C$, in which case the assumptions  $(A1), (A2)$ are vacuously true, or otherwise we replace $I$ with $I' := \{i \in I: i\neq cd, \mbox{ for } c, d \in C-I\} \cup \top_{I}$, where $\top_{I}$ a new element satisfying $(A1), (A2)$. Conditions $(A3), (A4)$ are implied by the last two compatibility conditions in the definition of a compatible quadruple. We can then construct the partial gluing $\alg B' \oplus_{\tau} \alg C'$, where $B' = B-F$, $C' = C-(F \cup I')$, $\tau = (\sigma_{F}, \gamma_{F},\ell_{I}, r_{I})$. 

\section{Variations of the constructions}
First, notice that the gluing constructions presented that involve a non-empty ideal $I$ work for both bounded and unbounded integral residuated lattices. In the case where the ideal is empty (and thus the construction is gluing over a congruence filter) one still obtains a new structure starting from $\mathsf{FL}_{w}$-algebras, but one of the two algebras is not a subalgebra with respect to $0$ anymore. 

Importantly, note again that as a special case of the gluing construction, where the filter is trivially the top element $\{1\}$ and the ideal is empty, we get the $1$-sum construction.
This also means that given any pair of integral residuated lattices $\alg B$ and $\alg C$ (with $\alg C$ having a lower bound or $1$ being join irreducible in $\alg B$) we can always glue them. We will call a gluing {\em trivial} if it is a $1$-sum, and  {\em non-trivial} otherwise.

\subsection{The congruence filter has a bottom element}
We first observe that if $F$ has a bottom element, then $\sigma_F$ and $\gamma_F$ have a very transparent definition: also the bottom element of $F$ multiplies and divides as the elements in $C-F$ in the gluing.

\begin{lemma}\label{lemma:botF}
Let $(\alg B, F)$ be a lower-compatible triple where $F$ has a bottom element $\bot_{F}$. Then given any $x \in B-F$, $\sigma_{F}(x)= \bot_{F} \cdot x = x  \cdot \bot_{F}$ and $\gamma_{F}(x) = \bot_{F} \back x = x /\bot_{F}$.
\end{lemma}
\begin{proof}
The proof can be directly extracted from the first paragraphs of the proof of Theorem \ref{prop:quadruplegluing}.
\end{proof}

 It turns out that assuming that $F$ has a bottom element is not a substantial restriction.
Given a lower compatible triple $(\alg B, F)$, where $F$ may or may not have a bottom element, and a new element $\bot_F$, we define the residuated lattice $\alg B_\bot$, where  for all $b \in B-F, f \in F$
\begin{itemize}
\item $b<\bot_F<f$,
\item  $\bot_F \cdot\bot_F = \bot_F \cdot f=f\cdot \bot_F=\bot_F$ and   $\bot_{F} \cdot b = b  \cdot \bot_{F}=\sigma_{F}(b)$,
\item	$\bot_{F} \back b = b /\bot_{F}=\gamma_{F}(b)$, $\bot_F \back f=f / \bot_F=1$, 
\item 	$b \back \bot_F=\bot_F / b=1$, and $f \back \bot_F=\bot_F / f=\bot_F$.
\end{itemize}

The following result is easy to prove and shows that every lower compatible triple can be embedded into one where the congruence filter has a bottom element.

 \begin{proposition}
If $(\alg B, F)$ is a lower compatible pair, 
then $(\alg B_\bot, F \cup \{\bot_F\})$ is also a lower compatible pair and $\alg B$ is a subalgebra of $\alg B_\bot$, except possibly for  join if $F$ has a bottom element.
\end{proposition}

\subsection{Non-linear order}
We are now going to study whether the order conditions for the congruence filter $F$ and the ideal $I$ can be weakened. We previously required $F$ to be strictly above all elements of $\alg B$ and $I$ to be strictly below all elements of $\alg C$. This ensures that the product is well-defined and respect the join operation. Notice that, in particular, asking $F$ to be strictly above all other elements implies that if there are elements $x, y \in B-F$ such that $x \lor y \in F$, then $F$ has a bottom element $\bot_{F}$ and $x \lor y = \bot_{F}$.  Lemma~\ref{lemma:botF} shows that the product defined in Theorem \ref{prop:quadruplegluing} preserves this particular kind of join. 

 We define a \emph{non-strict} lower-compatible pair a pair $(\alg B, F)$ that is lower-compatible except that $F$ is not strictly above all elements in $F$, but whenever $x, y \in B-F$ are such that $x \lor y = z \in F$, the element $z$ is such that $zb = bz = \sigma_{F}(b)$ and $z \back b = b  / z = \gamma_{F}(b)$ for all $b \in B-F$. 

Similarly, if $z, w \in C-I$ are such that $z \land w \in I$, then we required $I$ to have a top element $\top_{I}$ and $z \land w = \top_{I}$; this does not create issues with respect to the operations, given Lemma \ref{bounded-lemma}. We can then extend the construction to also include lattice ideals that are not strictly below all other elements. In particular, we call a \emph{non-strict} upper-compatible pair a pair $(\alg C, I)$ that is upper-compatible except that $I$ is not strictly below all other elements in $C$. 

Let us call a \emph{non-strict compatible quadruple} a compatible quadruple $(\alg B, F, \alg C, I)$ where the upper and lower-compatible pairs may be non-strict, by redefining all joins of elements $x, y \in B-F$ such that $x \lor_{B} y \in F$ to be the bottom element of $\alg C$, $x \lor y = \bot_{C}$, and all meets of elements $z,w \in C-I$ such that $x \land_{C} w \in I$ to be the top element of $B- F$, $x \land w = \top_{B}$. 

\begin{proposition}
If $(\alg B, F, \alg C, I)$ is a non-strict compatible quadruple, then $\alg B \oplus_{\pp} C$ is the gluing of $\alg B$ and $\alg C$ over $F$ and $I$. 
\end{proposition}
\begin{proof}
The proof of Theorem \ref{prop:quadruplegluing} can be adapted to this case. 
 \end{proof}

\section{Preservation}
In this section we will investigate the interaction of  the $(F-I)$-gluing construction with class operators and equations that are preserved.

\subsection{Preservation of identities}
We identify equations that are preserved by the $(F-I)$-gluing.
It is worth noticing that the gluing construction preserves commutativity. Moreover, we identify the cases when divisibility and semilinearity are preserved; linearity is obviously always preserved. 

\emph{Semilinear} integral residuated lattices (i.e., subdirect products of totally ordered integral residuated lattices) constitute a variety, axiomatized by the equation: 
$$[u\back(y\back x)u] \lor [v(x \back y)/v]= 1.\qquad{\mbox{(sl)}}$$
This equation characterizes semilinearity also in $\mathsf{FL_{w}}$-algebras. 
 In commutative subvarieties of $\mathsf{IRL}$ and $\mathsf{FL_{w}}$ semilinearity is characterized by the simpler prelinearity identity, obtained from (sl) by taking $u=v=1$:
$$(y \back x) \lor (x \back y ) = 1.\qquad{\mbox{(prel)}}$$

Commutative prelinear $\mathsf{FL_{w}}$-algebras are called MTL-algebras since are the equivalent algebraic semantics of Esteva and Godo's {\em Monoidal t-norm based logic}, the logic of left-continuous t-norms \cite{EstevaGodo2001}.
A residuated lattice $\alg A$ is called \emph{divisible} if the lattice order coincides with the inverse divisibility order: 
$$
a \le b \qquad\text{if and only if}\qquad \text{there are $ c,d \in A$ with $a =bc$ and $a=db$}.
$$
Divisibility is characterized equationally by:
$$
x \land y = x(x \back (x \land y)) = ((x \land y) / x)x \qquad{\mbox{(div)}}
$$
The latter in integral structures reduces to: $x \land y = x(x \back y) = (y/x)x$.
Semilinear, commutative and divisible $\mathsf{FL_{ew}}$-algebras 
are called BL-algebras and we denote their variety by $\mathsf{BL}$;  semilinear and divisible CIRLs are called basic hoops, and we refer to their variety by $\mathsf{BH}$. BL-algebras are the equivalent algebraic semantics of {\em H\'ajek's Basic Logic} \cite{H98}.

\begin{proposition}
If $\alg B \oplus_{P} \alg C$ is the $P$-gluing of the IRLs $\alg B$ and $\alg C$, where $P = (F, I)$, then:
\begin{enumerate}
\item  $\alg B \oplus_{P} \alg C$ is commutative iff both $\alg B$ and $\alg C$ are commutative.
\item  $\alg B \oplus_{P} \alg C$ is divisible iff both $\alg B$ and $\alg C$ are divisible, $\alg C$ has no $I$-divisors, and $ B = ((B - F) \cup \{1\}) \oplus_1 F$. 
\item  If $F \not = \{1\}$, then  $\alg B \oplus_{P} \alg C$ is semilinear iff both $\alg B$ and $\alg C$ are semilinear. If $F = \{1\}$ and $C^- \not= \emptyset$, then  $\alg B \oplus_{P} \alg C$ is semilinear iff  $\alg B$ is linear and $\alg C$ is semilinear. (If $F = \{1\}$ and $C^-= \emptyset$, then $\alg B \oplus_{P} \alg C= \alg B$.)
\end{enumerate}
\end{proposition}
\begin{proof}
Recall that $B^{-} = B-(F\cup I)$ and $C^{-} = C- (F \cup I)$. For readability we write $\sigma$ for $\sigma_{F}$, $\gamma$ for $\gamma_{F}$, $\ell$ for $\ell_{I}$ and $r$ for $r_{I}$.
\begin{enumerate}
\item $B$ and $C$ are closed under multiplication. Also, $bc=\sigma(b)=cb$ for $b \in B^-, c \in C^{-}$; thus the construction preserves commutativity. 
\item  Recall that in integral structures divisibility states that for all $x, y$, $$x \land y = x(x \back y) = (y/x)x.$$ Since $B$ and $C$ are closed under $\wedge, \cdot, \back, /$, the divisibility of $\alg B$ and $\alg C$ is a necessary condition for the gluing to be divisible. Note that if $x \in B^{-},$ and $y \in C^{-}$, then $x(x \back y) = x1 = x = x \land y$, and similarly $ (y/x)x = x = x \land y$, so divisibility holds in this case.

If $x \in C^{-}$ and $y \in B^{-}$, then $x \land y = y$, while $x(x\back y)$ and $(y / x)x$ depend on whether $x$ is an $I$-divisor or not.
Notice that if $x$ is a left $I$-divisor, then $ x \ell(x) \neq y$ since $\alg C$ is a subalgebra, therefore $$x(x\back y) = x \ell(x) \neq y = x \land y.$$
Similarly, if $x$ is a right $I$-divisor
$$(y/x)x = r(x)x \neq y = x \land y$$
Thus for the gluing to be divisible, $\alg C$ must have no $I$-divisors. Now, if $x$ is not an $I$-divisor, we get:
$$x(x\back y) =  x \gamma(y)= \sigma(\gamma(y)) = \sigma(y)$$ and similarly, $$(y/x)x =  \gamma(y) x= \sigma(\gamma(y)) = \sigma(y).$$
Thus for the gluing to be divisible, we need that for all $y \in B^{-}$, $\sigma(y) =y$. 

Notice that the same holds for $x \in C^{-}, y \in I$, by Lemma \ref{lemma:filtercompatible2}.
Consequently, for all $f \in F$, $y = \sigma(y) \leq fy\leq y$ and also $y = \sigma(y) \leq yf \leq y$, thus have that $fy = yf = y$. This implies that $\alg B$ is the $1$-sum $B = (B-F) \oplus_1 F$ (including the trivial case where $F = \{1\}$). 

Now we show that if $\alg B$ and $\alg C$ are divisible, $\alg C$ has no $I$-divisors and $ B = (B-F) \oplus_1 F$, then divisibility holds in the gluing. We only need to check the case where $x \in C^{-}, y \in B^{-}$, where we get $x \land y = y$and $$x(x\back y) = x \gamma(y)= \sigma(y) = (y/x) x$$ since $x$ is not an $I$-divisor. Now, if $ B = (B-F) \oplus_1 F$, all products between elements $f \in F$ and $x \in B-F$ are such that $fx = xf = x$. Thus, for $x \in B^{-}$ with $x \mathrel{\theta_{F}} y$, we have $x \back y, y \back x \in F$. So $x = x (x \back y) \leq y$ and $y = y (y \back x) \leq x$, hence $x = y$. Therefore, $\sigma(y) =\min[y]_{F}= y$ and divisibility holds in the gluing.

\item 
If $\alg B \oplus_{P} \alg C$ is semilinear, then $\alg C$ is semilinear, since it is a subalgebra except possibly for the meet. Also, in verifying semilinearity in $\alg B$, if $x,y, u, v \in B$, then $[u\back(y\back x)u]$, $[v(x \back y)/v] \in B$ and $[u\back(y\back x)u] \lor [v(x \back y)/v]= 1$ in $\alg B \oplus_{P} \alg C$. Since $[u\back(y\back x)u] \lor [v(x \back y)/v] \in B$, it follows that $[u\back(y\back x)u] \lor [v(x \back y)/v]= 1$ in $\alg B$; hence semilinearity holds in $\alg B$.

 Moreover, in the particular case where $\alg F = \{1\}$, let $a, b \in B$ be incomparable. Therefore, $a \not \leq b$ and $b \not \leq a$, so $1 \not = a \back b$ and $1 \not = b\back a$; hence $a \back b, b\back a \in B-F$, so $(a \back b) \lor (b \back a) \leq c$ for every $c \in C^-$. Since $C^- \not = \emptyset$ we have $c<1$ for some $c \in C^-$, so $(a \back b) \lor (b \back a) \leq c<1$. By the semilinearity of  $\alg B \oplus_{P} \alg C$, we get $(a \back b) \lor (b \back a)= 1$, a contradiction.

For the converse direction, suppose both $\alg B$ and $\alg C$ are semilinear. We check whether in $\alg B \oplus_{P} \alg C$ we have (sl):
$$[u\back(y\back x)u] \lor [v(x \back y)/v]= 1.$$

If all elements belong to $\alg C$, the equation follows from the semilinearity of $\alg C$. If $x$ and $y$ are comparable, then $y \leq x$ or $x \leq y$, so  $y\back x=1$ or $x \back y=1$; hence for all $u,v$,  $u\back (y\back x)u=1$ or $v(x \back y)/v=1$ and so (sl) holds. 

If $x \in B^-, y \in C^-$ or vice versa, then $x,y$ are comparable, so  (sl) holds. 
It remains to verify (sl) for  $x, y \in B^-$.

 If $F = \{1\}$ and $\alg B$ is linear, then $x,y$ are comparable, so (sl) holds. 

We now assume that $F \neq \{1\}$. For $x, y \in B^-$, if they are comparable, (sl) holds. If they are incomparable, then $y\back x \not =1$ and $x \back y \not =1$ and $y\back x , x \back y \in B$. The semilinearity of $\alg B$ yields
 $$(y\back x) \lor (x \back y) = 1.$$
Note that if $b_1, b_2 \in B$, $b_1 \not =1$, $b_2 \not = 1$ and $b_1 \lor b_2 =1$, then $b_1, b_2 \in F$. (If, say, $b_1 \not \in F$ and $b_2 \in F$, then since $B-F$ is strictly below $F$, we get  $b_1 \lor b_2 = b_2 \not = 1.$ If $b_1, b_2 \not \in F$, then since $B-F$ is strictly below $F \neq \{1\}$, there is $f \in F-\{1\}$ such that $b_1 \lor b_2 \leq f <1$.) The same holds for elements $c_1, c_2 \in C$.

Therefore, $y\back x , x \back y \in F$. If $u \in B$ (or $u \in C$), since $\alg B$ (respectively, $\alg C$) is semilinear, we can apply Lemma 6.5 in \cite{BlountTsinakis} and obtain that
$$[u\back(y\back x)u] \lor (x \back y)= 1.$$

If $[u\back(y\back x)u]=1$, then (sl) holds. If not, then $[u\back(y\back x)u]$ and $  (x \back y)$ are non-identity elements of $B$ ($C$, respectively) that join to $1$, so by above fact they are both in $F$. Given the semilinearity of $\alg B$ (respectively, $\alg C$), we can apply Lemma 6.5 in \cite{BlountTsinakis}, which states that whenever semilinearity holds, if $a \lor b = 1$, also $\gamma_1(a) \lor \gamma_2(b) = 1$, for any iterated conjugates $\gamma_1, \gamma_2$. With $v \in B$ (or $v \in C$), this yields precisely (sl).
\end{enumerate}
\end{proof}

Since the components of the gluing are subalgebras (except in the mentioned cases for the lattice operations), most one-variable equations are preserved.

\begin{proposition} The $P$-gluing $\alg B \oplus_{P} \alg C$, where $P= (F, I)$, of two IRLs preserves all one-variable equations not involving the lattice operations. Whenever $B-F$ is closed under joins, and $C-I$ is closed under meets, all one-variable equations satisfied by both $\alg B$ and $\alg C$ are preserved.
\end{proposition}
\begin{proof}
Follows directly from the definition of the operations.
\end{proof}
Thus, for example, the gluing construction preserves $n$-potency, $x^{n} = x^{n+1}$, for every $n \geq 1$. In fact, the gluing also preserves all monoidal equations, given the idempotency and the absorbing properties of the conucleus $\sigma$.
\begin{proposition} The $P$-gluing $\alg B \oplus_{P} \alg C$, with $P= (F, I)$, preserves all monoid equations valid in both $\alg B$ and $\alg C$.
\end{proposition}
\begin{proof}
If the equation has a variable that appears in only one side, then setting all the other variables equal to $1$, we obtain a consequence of the form $x^n=1$, for some $n \not =1$, and the only model of that equation is the trivial algebra. Therefore, we consider equations where all variables appear on both sides.
 
Since  $B$ and $C$ are closed under multiplication, if an equation holds in the gluing then it also holds in $\alg B$ and in $\alg C$. We now assume that some equation holds in $\alg B$ and in $\alg C$. If under some evaluation all variables are chosen from $B$ or all variables are chosen from $C$, then the equation holds true.  Now suppose that at least one variable is assigned to an element of $C^{-}$ and one variable is assigned to an element $B^{-}$. Assume that $X$, is the set of variables in the equation, $v$ is the evaluation, and that $X_C=\{x \in X: v(x)\in C^-\}$ corresponds to the variables that are mapped to elements of $C^-$; $X^c_C$ denotes the complement of $X_C$. We focus on the position of values $v(x)$, where $x \in X_C$, inside the equation; we group together the elements in between these $v(x)$ as follows. Given that $B$ is closed under multiplication, the evaluation of each side of the equation takes the form $$b'_1c_1b'_2c_2 \cdots b'_n,$$ where each $c_i$ is of the form $v(x)$, for some $x \in X_C$ and each $b'_i$ is a product of elements of the form $v(x)$, for certain $x \in X^c_C$, so $v(x)\in B$; hence $c_i \in C^-$ and $b'_i \in B$. 

By focusing on the elements adjacent to the $c_i$'s and using that  $cb = bc = \sigma(b) \in B-F$, for $c \in C^-$ and $b \in B-F$, and that $fc, cf \in C^{-}$ for  $c \in C^-$ and $f \in F$, the evaluation of the equation is reduced to a form that does not contain any elements of $C^-$ (recall that at least one variable is assigned to an element of $C^{-}$ and at least one variable is assigned to an element of $B^{-}$). Then, we use that $b_1\sigma(b_2) = \sigma(b_1b_2) =  \sigma(b_1)b_2$, for $b_1,b_2 \in B-F$, and that $f\sigma(b) = \sigma(b)f = \sigma(b)$, for $b \in B-F$ and $f \in F$, and the idempotency of $\sigma$. In the end, the evaluation of the equation takes the form $\sigma(b_1b_2 \cdots b_n)=\sigma(b_{n+1}b_{n+2} \cdots b_m)$, where the $b_i$'s are exactly the elements of the form $v(x)$, for $x\in X^c_C$, in the exact order they appear in the equation. Therefore, by substituting $1$ for all $x$ with $x\in X_C$, and the appropriate value $b_i$ for the other variables, the original equation (which is valid in $\alg B$), yields $b_1b_2 \cdots b_n=b_{n+1}b_{n+2} \cdots b_m$, so $\sigma(b_1b_2 \cdots b_n)=\sigma(b_{n+1}b_{n+2} \cdots b_m)$ is also valid.
\end{proof}

\subsection{HSP$_{U}$}
Constructions as the one presented in this paper are particularly interesting when they help us better understand and describe the structure theory of the algebras. In what follows, we will just call ``gluing'' a gluing over a filter and an ideal.
In this section we characterize when a gluing is subdirectly irreducible. We will also study the subalgebras, homomorphic images and ultrapowers of a gluing. 
We recall that  $\alg{Fil}(\alg A)$ denotes the lattice of congruence filters of $\alg A$.
\begin{proposition}\label{prop:latticefilters} Consider the $P$-gluing $\alg B \oplus_{P} \alg C$, with $P= (F, I)$, of two integral residuated lattices $\alg B$ and $\alg C$. We distinguish two cases:
\begin{enumerate}
\item If $C-I$ is a congruence filter of $\alg C$, then  $\alg{Fil}(\alg B \oplus_{\pp} \alg C)$ is isomorphic to $\alg{Fil}(\alg B) \oplus \alg{Fil}(\alg C-I)$, the poset ordinal sum of the two lattices.
\item Otherwise, $\alg{Fil}(\alg B \oplus_{\pp} \alg C) \cong \alg{Fil}(\alg C)$.
\end{enumerate}
\end{proposition}
\begin{proof}
The first claim follows from the definition of the order and operations in the gluing construction, see Figure \ref{fig:gluing1}. Now, if $C-I$ is not a congruence filter, then the congruence filter generated by $C-I$ has non-empty intersection with $I$, i.e., either some conjugate or some product of elements in $C^{-}$ are in $I$. Since $C$ is closed  under multiplication and divisions, this is also true in the gluing $\alg B \oplus_{P} \alg C$. Thus the congruence filter $\langle C - I\rangle$ generated by $C-I$ in the gluing also has nonempty-intersection with the ideal $I$. Since filters are closed upwards,  $B^{-}$ is contained in $\langle C - I\rangle$  and the second claim follows.
\end{proof}
\begin{corollary}
For $P= (F, I)$, if $F \neq \{1\}$, then $\alg B \oplus_{P} \alg C$ is subdirectly irreducible iff $\alg F$ is subdirectly irreducible iff $\alg B$ is subdirectly irreducible iff $\alg C$ is subdirectly irreducible. If $F=\{1\}$, then $\alg B \oplus_{P} \alg C$ is subdirectly irreducible iff $\alg C$ is subdirectly irreducible. 
\end{corollary}
\begin{proof}
Assume first that $F$ is not trivial. By Proposition \ref{prop:latticefilters} and standard universal algebraic results (see Theorem 8.4 in \cite{BS}), $\alg B \oplus_{\pp} \alg C$ is subdirectly irreducible iff $\alg F$ is subdirectly irreducible as an IRL, thus the claim follows since $F$ is a congruence filter of both $\alg B$ and $\alg C$, (strictly) above all their other elements.
For $F=\{1\}$, it follows from the definition of the operations and order of the gluing (see also Figure \ref{fig:gluing1}) that $\alg B \oplus_{P} \alg C$ is subdirectly irreducible iff $\alg C$ is subdirectly irreducible. 
\end{proof}
We can now describe the homomorphic images of a gluing. 
\begin{proposition}\label{prop:homogluing}
Let $h$ be a homomorphism having as domain a gluing  $\alg B \oplus_{P} \alg C$ and $H$ the associated congruence filter (the preimage of $1$). If $B/H, C/H, F/H, I/H$ denote the images under $h$, we have that
the homomorphic image of $\alg B \oplus_{P} \alg C$ via $h$ is isomorphic to
\begin{enumerate}
	\item\label{homo1} $\alg C/H$, if $H \cap I \not = \emptyset$.
	\item\label{homo2} $\alg B/H$, if $H \cap I = \emptyset$ and $H \cap B^- \not = \emptyset$.	
	\item\label{homo3} $\alg B/H \oplus_{P/H} \alg C/H$, where $P/H = (F/H, I/H)$, if $H \cap B^- = \emptyset$.
\end{enumerate}
\end{proposition}

\begin{proof} The fact that the homomorphic image through $h$ is given by the gluing $\alg B/H \oplus_{P/H} \alg C/H$ follows from Proposition \ref{prop:latticefilters}. Moreover, notice that (\ref{homo1}) and (\ref{homo2}) are particular cases of (\ref{homo3}). 
\end{proof}
We call a subalgebra $\alg S$ of $\alg C$ \emph{divisor-special} if whenever it contains an element $c$ that is a left $I$-divisor, it also contains $\ell_{I}(c)$, and similarly if $d \in S$ where $d$ is a right $I$-divisor, then also $r(d) \in S$. 
We call a subalgebra $\alg T$ of $\alg B$ \emph{$\sigma$-special} if for all $b \in T-F$, also $\sigma(b) \in T$, and \emph{$(\sigma,\gamma)$-special} if also $\gamma(b) \in T$.

\begin{proposition}\label{prop:subgluing}
Let $\alg B \oplus_{P} \alg C$, with $P = (F, I)$, be the $P$-gluing of IRLs $\alg B$ and $\alg C$. Then a subalgebra $\alg S$ of $\alg B \oplus_{P} \alg C$ is one of the following:
\begin{enumerate}
\item $\alg S$ is a subalgebra of $\alg C$ that does not include elements whose meet is the top element of $B-F$. 
\item $\alg S$ is a subalgebra of  $\alg B$ that does not include elements whose join is the bottom element of $C-I$.
\item A gluing $\alg B_{1} \oplus_{P_{1}} \alg C_{1}$, with $P_{1} = (F_{1}, I_{1})$, where: \begin{itemize}
\item $F_{1} \subseteq F, I_{1} \subseteq I$ and $F_{1} \cup I_{1}$ is a subalgebra of $F \cup I$;
\item $\alg C_{1}$ is a divisor-special subalgebra of $\alg C$ containing at least an element that is not an $I$-divisor; 
\item $\alg B_{1}$ is a (nonempty) special $(\sigma,\gamma)$-subalgebra of $\alg B$.
\end{itemize}
\item A gluing $\alg B_{2} \oplus_{P_{2}} \alg C_{2}$, with $P_{2} = (F_{2}, I_{2})$, where: \begin{itemize}
\item $F_{2} \subseteq F, I_{2} \subseteq I$ and $F_{2} \cup I_{2}$ is a subalgebra of $F \cup I$;
\item $\alg C_{2}$ is a divisor-special subalgebra of $\alg C$ containing only $I$-divisors; 
\item $\alg B_{2}$ is a (nonempty) $\sigma$-special subalgebra of $\alg B$.
\end{itemize}
\end{enumerate}
\end{proposition}
\begin{proof}
The first two claims follow from the fact that $\alg B$ and $\alg C$ are subalgebras except possibly for those joins and meets.
The other two claims follow from the definition of the operations. For example, whenever a subalgebra $\alg S$ of $\alg B \oplus_{P} \alg C$ contains both an element $b \in B^{-}$ and an element $c \in C^{-}$ that is not an $I$-divisor, then both the minimal, $\sigma(b_F)$, and the maximal, $\gamma_F(b)$, element of the equivalence class $[b]_{F}$  also belong to $\alg S$, since $bc=cb=\sigma(b)$ and $c \back b = b/c = \gamma(b)$. Moreover, if in the subalgebra $\alg S$ there is at least an element $b \in B^{-}$, and $c \in C^{-}$ is a left $I$-divisor, then $c \back b = \ell(b)$ and similarly if $d \in C^{-}$ is a right $I$-divisor then $b/d = r(d)$, thus such elements need to be in $S$.
\end{proof}

We are now going to show that an ultrapower of a gluing $\alg B \oplus_{P} \alg C$ is a gluing of ultrapowers of $\alg B$ and $\alg C$. 
\begin{proposition}
$P_{U}(\alg B \oplus_{P} \alg C) \subseteq P_{U}(\alg B) \oplus_{(P_{U}(F), P_U(I))}  P_{U}(\alg C)$
\end{proposition}
\begin{proof}
We sketch the proof.
Let $\alg A = \prod_{j \in J} \alg B \oplus_{P} \alg C$, and let $U$ be an ultrafilter on $J$. For $x=(x_j)_{j\in J} \in A$ and $j \in J$, we distinguish cases according to whether
$x_{j}$ is in $B^{-}, C^{-}$, $F$ or $I$, and partition $J$ in the sets:
$$ J_{B}(x)  = \{j \in J: x_{j} \in B^{-}\},\, J_{C}(x)  = \{j \in J: x_{j} \in C^{-}\},$$$$ J_{F}(x)  = \{j \in J: x_{j} \in F\}, J_{I}(x)  = \{j \in J: x_{j} \in I\}.$$
 Since $U$ is an ultrafilter, for each $x \in A$ only one of these sets belongs to $U$; also if $[x]_U=[y]_U$, then the corresponding sets are both in $U$ or neither in $U$. This allows us to define the sets
$$B_U=\{[x]_U: J_B(x) \in U\}, C_U=\{[x]_U: J_C(x) \in U\},$$
$$ F_U=\{[x]_U: J_F(x) \in U\}, I_U=\{[x]_U: J_I(x) \in U\}.$$
It is easy to see that $B_U \cup F_U \cup I_U$ and $C_U \cup F_U \cup I_U$ are IRLs with the inherited operations, that $(B_U \cup F_U \cup I_U, F_U, C_U \cup F_U \cup I_U, I_U)$ is a compatible quadruple and that $B_U \cup F_U \cup I_U\in P_{U}(\alg B)$ and $C_U \cup F_U \cup I_U \in P_{U}(\alg C)$. Finally, it can also be shown that $\alg A/U$ is isomorphic to the gluing $(B_U \cup F_U \cup I_U)/U \oplus_{(F_U, I_U)} (C_U \cup F_U \cup I_U)/U$ (see Proposition 3.3 in \cite{AM} for a similar instance).
\end{proof}
\section{Amalgamation property}
The gluing construction can be seen as a way of finding a (strong) amalgam of two algebras $\alg B$ and $\alg C$ in the particular case where the common subalgebra $\alg A$ corresponds to the union of a congruence filter and an ideal of both $\alg B$ and $\alg C$.

 
%
More precisely, let $(\alg B, F, \alg C, I)$ be a compatible quadruple, and let us call $\alg P$ the subalgebra of $\alg B$ (equivalently, $\alg C$) with domain $F \cup I$. Then $\alg P$ embeds into both $\alg B$ and $\alg C$; let us name the embeddings with $i, j$ respectively. Moreover, by construction both $\alg B$ and $\alg C$ embed in the gluing $\alg B \oplus_{P} \alg C$. Let us denote these embeddings  by $h, k$, respectively. With this notation in mind:
%
%
\begin{proposition}
Let $(\alg B, F, \alg C, I)$ be a compatible quadruple as above, and let $\alg P$ be the subalgebra of $\alg B$ (equivalently, $\alg C$) with domain $F \cup I$. $(\alg B \oplus_{P} \alg C, h, k)$  is a (strong) amalgam of $(\alg P, \alg B, \alg C, i, j)$.
\end{proposition}

In this section we present two applications of the gluing and partial gluing constructions, respectively, where the gluing constructions shed some light on when amalgamation holds in classes of (bounded) IRLs.

\subsection{Generalized rotations}
We observe that  the generalized $n$-rotation construction introduced in \cite{BMU18} is actually an example of gluing and we generalize this construction to the non-commutative case, using the gluing perspective. 

The \emph{generalized $n$-rotation}, for $n \geq 3$, defined in \cite{BMU18} is itself inspired by ideas in Wro\'nski's reflection construction for BCK-algebras \cite{W83} and generalizes in this context the (dis)connected rotation construction developed by Jenei  \cite{Je00,Je03} for ordered semigroups. In these constructions given a CIRL, and also more generally in \cite{GR04} given a topped residuated lattice (not necessarily commutative or integral), a bounded involutive structure is produced, obtained by attaching below the original CIRL a rotated copy of it. On the other hand, the generalized rotation takes an CIRL and generates a bounded CIRL, which is not necessarily involutive, by attaching below it a rotated (possibly proper) \emph{nuclear image} of the original.  The generalized $n$-rotation, for $n \geq 3$,  further adds a \L ukasiewicz chain of $n$ elements, $n-2$ of which are between the original structure and its rotated nuclear image (see Figure \ref{fig:rotation} for a sketch). 

We introduce the non-commutative version of the generalized $n$-rotation, building on the construction in \cite{GR04}, and we apply it to IRLs.

We first recall from \cite{GR04} that the disconnected rotation  of an IRL $\alg A$ is the $\sf{FL}_{ew}$-algebra $\alg A^*$ whose lattice reduct is given by the  union of $A$ and its disjoint copy $A' = \{a' : a \in A\}$ with dualized order, placed below $A$: for all $a, b \in A$,  $$ a' < b, \mbox{ and } a' \leq b' \mbox{ iff } b \leq a.$$
In particular, the top element of $\alg A^*$ is the top $1$ of $\alg A$ and the  bottom element of $\alg A^*$ is the copy $0 := 1'$ of the top $1$. 
$\alg A$ is a subalgebra, the products in $A'$ are all defined to be the bottom element $0=1'$, and furthermore, for all $a,b \in A$, $$a\cdot b' = (b/a)', \quad b'\cdot a = (a \back b)';$$ 
$$a \back b' = a' /b = (b\cdot a)', \quad a' \back b' = a/b, \quad b'/a' = b \back a.$$
A nucleus on a residuated lattice $\alg A=(A, \wedge, \vee, \cdot, \, \back, 1)$ is a closure operator $\delta$ on $\alg A$ that satisfies $\delta(x) \delta(y) \leq \delta(xy)$, for all $x,y \in A$. It is known that then $\alg A_\delta=(\delta[A],  \wedge, \vee_\delta, \cdot_\delta, \, \back, \delta(1))$ is a residuated lattice, where $x \vee_\delta y=\delta(x \vee y)$ and  $x \cdot_\delta y=\delta(xy)$.

The  \emph{generalized disconnected rotation} $\alg A^\delta$ of a IRL $\alg A$ with respect to a nucleus $\delta$ on $\alg A$ serves as a non-commutative version of the construction given in \cite{AFU} (which in turn was inspired by \cite{CT06}). It differs from the disconnected rotation above in that it replaces $A'$ with $\delta[A]' = \{\delta(a)': a \in A\}$, where $\delta(a)'$ is short for $(\delta(a))'$. It is easy to see that then with respect to the above order we have 
$\delta(a)' \land \delta(b)' = \delta(\delta(a) \lor \delta(b))'$.
 Moreover, for all $a \in A$, $b \in \delta[A]$, $$a \back b'  = (\delta(b a))'\qquad  \mbox{ and } \qquad b'/a= (\delta(a b))'.$$

The proof that $\alg A^\delta$ is a residuated lattice is a very small variation of the analogous proof for $\alg A^*$, given in Section 6 of \cite{GR04}. In particular, the product is well-defined given the fact that $a \back b' = a'/b = (\delta(b \cdot a))'$. For an in-depth analysis of this and other rotation constructions, see \cite{GalatosCastano}. 

Clearly, the disconnected rotation is the special case of a generalized disconnected rotation where the nucleus is the identity map. 
Now, the \emph{generalized $n$-rotation} $\alg A^\delta_n$ of an IRL $\alg A$ with respect to a nucleus $\delta$ and $n \geq 3$ is defined on the disjoint union of $A^\delta$ and $\{\ell_i: 0<i<n-1\}$. We also set $\ell_0=0$ and $\ell_{n-1}=1$, the bounds of $\alg A^\delta$. The order extends the order of  $\alg A^\delta$ by 
$$b <  \ell_{1} < \ldots < \ell_{n-2} < a,$$
 for all $a \in A$ and $b \in \delta[A]$; see the rightmost structure of Figure \ref{fig:rotation}. The operations extend those of $\alg A^\delta$, of the $n$-element \L ukasiewicz chain \textbf{\L}$_{n}$, where $0 = \ell_{0} < \ell_{1} < \ldots < \ell_{n-2} < \ell_{n-1} = 1$, and  for $0<i<n-1$: $$a \ell_{i} = \ell_{i} = \ell_{i}a, \quad b'\ell_{i} = 0 = \ell_{i}b'.$$
The proof that the resulting structure is an $\mathsf{FL_{ew}}$-algebra is an easy combination of the proofs of \cite{GR04} and \cite{BMU18}, but it also follows from Proposition \ref{lemma:ngrot} below. We mention that in \cite{BMU18} the generalized $n$-rotation is defined with respect to nuclei that preserve the lattice operations, thus the construction we propose here is more general also in the commutative case. The subvariety $\mathsf{MVR_{n}}$ of $\mathsf{FL_{ew}}$ generated by the generalized $n$-rotations of CIRLs where the nuclei preserve the lattice operations is axiomatized in \cite{BMU18}. This class of algebras contains as subvarieties, among others, the varieties of: G\"odel algebras, product algebras, the variety generated by perfect MV-algebras, nilpotent minimum algebras, $n$-contractive BL-algebras, and Stonean residuated lattices.

We now show that the generalized $n$-rotation is a special case of a gluing.  We refer to $1$-sums of the kind \textbf{\L}$_{n} \oplus_1 \alg A$ as \emph{$n$-liftings }of an IRL $\alg A$. Then generalized $n$-rotations are gluings of disconnected rotations and $n$-liftings.
\begin{proposition}\label{lemma:ngrot}
The generalized $n$-rotation $\alg A^\delta_n$ of an IRL $\alg A$ with respect to a nucleus $\delta$ for $n \geq 3$ is isomorphic to the gluing $\alg A^\delta \oplus_{(A,\{0\})} (\textbf{\L}_{n} \oplus_1 \alg A)$ of the generalized disconnected rotation $\alg A^\delta$ and the $1$-sum \textbf{\L}$_{n} \oplus_1 \alg A$ over $\alg A$, $\{0\}$. 
\end{proposition} 
\begin{proof}
First we show that the conditions of the gluing are satisfied.
Note that $A$ is a congruence filter strictly above all other elements of 
$\alg A^\delta$ and \textbf{\L}$_{n} \oplus_1 \alg A$, and $\{0\}$ is a shared lattice ideal. 
 For all $x \in A^\delta-A=\delta[A]'$, there is $y \in A$ with  $x = \delta(y)'$, so 
$x \back 0 = (\delta(y))' \back 1' = \delta(y) /1 = \delta(y) \in A$; also 
 $0 \back x = 1 \in A$. Therefore, $x \mathrel{\theta_A} 0$ and $\sigma_{A} (x)= \min[x]_{A} = 0$. Furthermore, since all elements in (\L$_{n} \oplus_1 \alg A) - (A \cup \{0\})$ are $0$-divisors, $\gamma$ does not need to be defined. Moreover, since $\sigma_A(x) = 0$, for all  $x \in A^\delta-A$, $\sigma$ is clearly absorbing. Thus $(\alg A^\delta, A)$ is a weak lower-compatible pair. 

To show that (\L$_{n} \oplus_1 \alg A, \{0\})$ is an upper-compatible pair, first note that $\{0\}$ is a lattice ideal strictly below all other elements. Moreover, since being an $I$-divisor here means being a $0$-divisor, we have $\ell(x) = x \back 0$ and $r(x) = 0/x$, for all $x \in$ (\L$_{n} \oplus_1 A) - (A \cup\{0\})=$ \L$_n-\{0, 1\}$.

Now we prove that $(\alg A^\delta, A,$ \textbf{\L}$_{n} \oplus_1 \alg A, \{0\})$ is a compatible quadruple:
\begin{enumerate}
\item All elements of \L$_n - \{1\}$ are 0-divisors, thus the first condition is satisfied. 
\item If $c,d \in $ \L$_{n}  -  \{0\}$, with $cd = 0$, then $0x = x0 =0= \sigma_A(x)$ for all $x \in \delta[A]' - \{0\}$. 
\item $\delta[A]'$ is closed under join, so this condition is vacuously true.
\item  (\L$_{n} \oplus_1 A) -  \{0\}$ has a least element $\ell_1$, so this condition is also vacuously true.
\end{enumerate}
It is clear that $\alg A^\delta$ and \textbf{\L}$_{n} \oplus_1 \alg A$ are subalgebras of the gluing and also of the generalized $n$-rotation and they are ordered the same way in both of these structures. Finally, for  $\ell_i \in$ \L$_{n}-\{1, 0\}$, $1<i<n-1$ and $\delta(a)' \in \delta[A]'$, we have $\ell_i \delta(a)' = 0$ in both the generalized $n$-rotation and in the gluing.
\end{proof}

\begin{figure}
\begin{center}
\begin{tikzpicture}
\footnotesize{
 \draw (0.5,1.2) arc (0:180: 0.5 and 1.1);
 \draw [dotted] (0.5, 1.2) -- (-0.5, 1.2);
  \draw (0.5,0.4) arc (0:-180: 0.5 and 1.1);
 \draw [dotted] (0.5, 0.4) -- (-0.5, 0.4);
  \fill (0,2.3) circle (0.05);
    \fill (0,-0.7) circle (0.05);
 \node at (0.15,2.5) {$1$}; 
  \node at (0.35,-1) {$1' = 0$};  
  \node at (1.5,1) {$\oplus_{\pp}$};
 \draw (3.5,1.2) arc (0:180: 0.5 and 1.1);
 \draw [dotted] (3.5, 1.2) -- (2.5, 1.2);
  \fill (3,2.3) circle (0.05);
  \fill (3,0.5) circle (0.05);
    \fill (3,0.3) circle (0.05);
    \fill (3,0.1) circle (0.05);
        \fill (3,-0.1) circle (0.05); 
            \fill (3,-0.3) circle (0.05);
                \fill (3,-0.5) circle (0.05);   
     \fill (3,-0.7) circle (0.05);
  \node at (3,-1) {$0$}; 
    \node at (3.1,2.5) {$1$};   
  \node at (4.5,1) {${=}$};  
  \draw (6.5,1.7) arc (0:180: 0.5 and 1.1);
 \draw [dotted] (6.5, 1.7) -- (5.5, 1.7);
  \fill (6,2.8) circle (0.05);
  \fill (6,1.1) circle (0.05);
    \fill (6,0.9) circle (0.05);
    \fill (6,0.7) circle (0.05);
        \fill (6,0.5) circle (0.05); 
            \fill (6,0.3) circle (0.05);
                \fill (6,1.3) circle (0.05);  
 \draw (6.5,-0.1) arc (0:-180: 0.5 and 1.1);
 \draw [dotted] (6.5, -0.1) -- (5.5, -0.1);
  \node at (6.1,3) {$1$}; 
    \node at (6,-1.4) {$0$}; 
\fill (6,-1.2) circle (0.05);  

    }
 \end{tikzpicture}\end{center} \caption{The gluing of the disconnected rotation of an IRL $\alg A$ with its $8$-lifting: $\alg A^{*} \oplus_{\pp } (${\bf\L}$_{8} \oplus_1 \alg A)$, where $\pp = (\alg A, 0)$.} \label{fig:rotation}
 \end{figure}
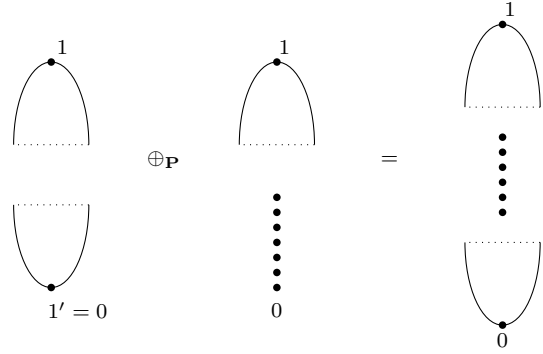

In \cite[Theorem 3.11]{AU19}, it is shown that a variety $\vv V$ of semilinear CIRLs has the AP if and only if the variety generated by generalized $n$-rotations of chains in $\vv V$ with respect to a nucleus definable by a term in the language of residuated lattices, has the AP. This allows to transfer known results about the AP in relatively tame varieties of CIRLs, to varieties of $\mathsf{FL_{ew}}$-algebras that are more complicated to study. For instance, since basic hoops, Wajsberg hoops, cancellative hoops, G\"odel hoops, all have the AP, so do the varieties generated by their generalized $n$-rotations (\cite{AU19}, Corollary 3.12).

We show that one can go in the same direction also in the non-commutative case, with the following bridge result. 
\begin{proposition}\label{lemma:amalgam}
Let $\delta$ be a term-defined nucleus for a class $\mathsf{K}$ of IRLs, and let $n \geq 3$. The following are equivalent:
\begin{enumerate}
	\item $\vv K$ has the amalgamation property; 
	\item the class of generalized $n$-rotations $\{\alg A^{\delta}_{n}: \alg A \in \vv K\}$ has the amalgamation property.
	\item the class $\{\alg A^{\delta}_{m}: \alg A \in \vv K, m-1 \mbox{ divides } n -1\}$ has the amalgamation property.
	\end{enumerate} 
\end{proposition}
\begin{proof}
In this proof, let us denote with $\vv  K^{\delta}_{n}$ the class of the generalized $n$-rotations via $\delta$ of algebras in $\vv K$.
We first show (1) $\Leftrightarrow$ (2).
The key idea here is that homomorphisms of generalized $n$-rotations are uniquely determined by their restriction to the upper-compatible triple in the gluing: i.e., the IRL and the \L$_{n}$ chain. This is due to the fact that the lower-compatible triple is a rotation whose domain is $A \cup \delta[A]'$ for some IRL $\alg A$, and the elements $\delta(a)' \in A$ are such that $\delta(a)' = a \back 0$, thus their homomorphic images are determined by those on $A$. 

More precisely, any homomorphism $h: A \to B$, for $\alg A, \alg B$ IRLs, extends to a homomorphism $\bar h: \alg A^\delta_{n} \to \alg B^\delta_{n}$ in the following way: $$\bar h(a) = h(a), \quad \bar h(\delta(a)') = \delta(h(a))'\quad \bar h(l_{i}) = l_{i} \mbox{ for all } i: 1 \ldots n-1$$
And vice versa, given any homomorphism $k: \alg A^\delta_{n} \to \alg B^\delta_{n}$ the restriction $k_{A}$ to $\alg A$ is a homomorphism from $A$ to $B$. Indeed, given $a \in A$, suppose $k(a) \in \alg B^\delta_{n} - B$. Then $k(a^{n}) = k(a)^{n} = 0_{\alg B^\delta_{n}}$, but $a^{n} \in A$, since $A$ is a congruence filter of the disconnected rotation. This leads to a contradiction, since $\neg (a^{n}) = a^{n} \back 0 = \delta(a^{n})'$, thus $ (\neg (a^{n}))^{2} = 0$, but $$k((\neg (a^{n}))^{2}) = (\neg (k(a^{n})))^{2} = (\neg 0)^{2} = 1^{2} = 1 \neq k(0) = 0.$$
So, $a \in A$ implies $k(a) \in B$. Moreover, $h$ is an embedding iff $\bar h$ is an embedding, and if $k$ is an embedding then clearly $k_{A}$ is an embedding.

Thus, suppose that $\vv K$ has the amalgamation property, and consider a V-formation in $\vv  K^{\delta}_{n}$: $\alg A^\delta_{n}, \alg B^\delta_{n}, \alg C^\delta_{n}$ with embeddings $i: \alg A^\delta_{n}\to \alg B^\delta_{n}, j: \alg A^\delta_{n}\to \alg C^\delta_{n}$. Then one can consider the restrictions of $i$ and $j$ to $\alg A$, and obtain a V-formation in $\vv K$, given by $\alg A, \alg B, \alg C$ and the embeddings $i_{A}, j_{A}$. This has an amalgam, say $\alg D$ with embeddings $f: B \to D, g : C \to D$ such that $f \circ i = g \circ j$. Thus it follows from what was shown before that $\alg D^\delta_{n}$ is going to be an amalgam for the V-formation in $\vv  K^{\delta}_{n}$, with embeddings $\bar f, \bar g$. 

Similarly, supposing that $\vv  K^{\delta}_{n}$ has the amalgamation property, we consider a V-formation in $\vv K$: $\alg A, \alg B, \alg C \in \vv K$ and embeddings $k:A \to B, l:A\to C$. We take the corresponding V-formation in $\vv  K^{\delta}_{n}$: $\alg A^\delta_{n}, \alg B^\delta_{n}, \alg C^\delta_{n}$ with embeddings $\bar k, \bar l$. The amalgam in $\vv  K^{\delta}_{n}$ is going to be some $\alg D^\delta_{n}$, with $\alg D$ an IRL, and embeddings $s:  \alg B^\delta_{n} \to \alg D^\delta_{n}, t: \alg C^\delta_{n} \to \alg D^\delta_{n}$. Thus $\alg D$ with embeddings $s_{B}, t_{C}$ are an amalgam for the V-formation in $\vv K$. Therefore, (1) and (2) are equivalent.

 While (3) clearly implies (2), (2) $\Rightarrow$ (3) can be shown again via the fact that homomorphisms of generalized $n$-rotations are uniquely determined by their restriction to the upper-compatible triple in the gluing. In particular, consider a V-formation $\alg A^\delta_{j}, \alg B^\delta_{k}, \alg C^\delta_{l}$, with $\alg A^\delta_{j}$ embedding in the other two. If the amalgam in $\vv K$ of $\alg A, \alg B, \alg C$ is $\alg D$, then it is routine to check that the desired amalgam is given by $\alg D^\delta_{\textrm{lcm\{k,l\}}}$.
\end{proof}
\subsection{A $2$-potent variety of $\sf{FL}_{ew}$-algebras}
We are now going to study a variety in which the subdirectly irreducible members can be characterized as partial gluings of a class of algebras. More precisely, as partial gluings of  simple 2-potent (i.e., satisfying $x^{2} = x^{3}$) CIRL-chains.
Let us consider totally ordered $\sf{FL}_{ew}$-algebras that are 2-potent, and such that for every $x, y$
$$x = 1 \mbox{ or } x \cdot (x \land y) \leq (x \land y)^{2}.$$
Equivalently, for $y \leq x <1$, we have $xy=y^2$.
The above is a positive universal first-order formula, thus by results in \cite{Ga2004} these structures generate a variety of MTL-algebras that satisfy $x^{2} = x^{3}$ and $$x \lor\; ((x \cdot (x \land y)) \backslash (x \land y)^{2}) = 1.$$
We will call this variety of $\mathsf{FL}_{ew}$-algebras $\mathsf{GL_{2}}$.
One can easily see that the finite chains in this variety are of the form shown in Figure \ref{fig:2potent}. 

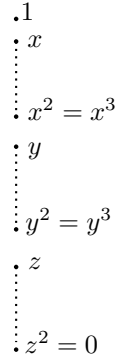
\begin{figure}[h!]
\begin{center}
\begin{tikzpicture}
 \fill (0,3) circle (0.03);
 \node at (0.15,3.1) {$1$}; %
 \fill (0,2.7) circle (0.03);
 \node at (0.25,2.7) {$x$};%
  \fill (0,1.7) circle (0.03);
  \node at (0.75,1.8) {$x^{2} = x^{3}$}; %
 \draw [dotted,thick] (0,2.7) -- (0,1.7);%
  
   \fill (0,1.3) circle (0.03);
 \node at (0.25,1.25) {$y$}; %
  \fill (0,0.2) circle (0.03);
  \node at (0.7,0.3) {$y^{2} = y^{3}$}; %
 \draw [dotted,thick] (0,1.2) -- (0,0.2);%
  
   \draw [dotted,thick] (0,-0.3) -- (0,-1.3);
    \node at (0.25,-0.25) {$z$};
   \fill (0,-0.3) circle (0.03);
    \fill (0,-1.4) circle (0.03);
    \node at (0.6,-1.3) {$z^{2} = 0$};%
 \end{tikzpicture}\end{center} \caption{A subdirectly irreducible algebra in $\mathsf{GL}_{2}$. }\label{fig:2potent}\end{figure}

More in detail, we will show that any $\mathsf{GL_{2}}$-chain consists of subintervals made of simple 2-potent CIRL-chains, and it can be characterized as \emph{partial gluings} of such chains. Using this representation, we will show that the amalgamation property fails for the class of $\mathsf{GL_{2}}$-chains. In particular, using their representation as partial gluings will allow us to fully characterize their subalgebras and determine exactly which V-formations have an amalgam.

First, let us show how we can iterate the partial gluing construction in this case. For example, we consider three  
simple 2-potent CIRL-chains $\alg S_{1}, \alg S_{2}$ and $\alg S_{3}$, as in Figure \ref{fig:simplegluings}.

\begin{figure}
\begin{center}
\begin{tikzpicture}
{\footnotesize
   \fill (0,3) circle (0.03);
 \node at (0.15,3.1) {$1$}; 
 \fill (0,2.7) circle (0.03);
 \node at (0.25,2.7) {$x$}; 
  \fill (0,1.7) circle (0.03);
  \node at (0.6,1.8) {$x^{2} = 0$}; 
    \node at (0.1,1.3) {$\alg S_{1}$}; 
 \draw [ dotted,thick] (0,2.7) -- (0,1.7);
  
 \fill (2,3) circle (0.03);
 \node at (2.15,3.1) {$1$}; 
 \fill (2,2.7) circle (0.03);
 \node at (2.25,2.7) {$y$}; 
  \fill (2,1.7) circle (0.03);
  \node at (2.6,1.8) {$y^{2} = 0$}; 
    \node at (2.1,1.3) {$\alg S_{2}$}; 
 \draw [ dotted,thick] (2,2.7) -- (2,1.7);

  \fill (4,3) circle (0.03);
 \node at (4.15,3.1) {$1$}; 
 \fill (4,2.7) circle (0.03);
 \node at (4.25,2.7) {$z$}; 
  \fill (4,1.7) circle (0.03);
  \node at (4.6,1.8) {$z^{2} = 0$}; 
 \draw [ dotted,thick, ] (4,2.7) -- (4,1.7);
  \node at (4.1,1.3) {$\alg S_{3}$};


   \fill (6.7,3) circle (0.03);
 \node at (6.85,3.1) {$1$}; 
 \fill  (6.7,2.7) circle (0.03);
 \node at (6.95,2.7) {$x$}; 
  \fill (6.7,1.7) circle (0.03);
  \node at (7.2,1.8) {$x^{2} = 0$}; 
    \node at (6.8,1.3) {$\alg S_{1}$}; 
 \draw [ dotted,thick, ] (6.7,2.7) -- (6.7,1.7);  
 
    \node at (8.1,2.5) {$ \oplus_{\tau_{2}}$};
 
   \fill (8.8,3) circle (0.03);
 \node at (8.95,3.1) {$1$}; 
 \fill (8.8,2.7) circle (0.03);
 \node at (9.05,2.7) {$y$}; 
  \fill (8.8,1.7) circle (0.03);
  \node at (9.05,1.8) {$y^{2}$}; 
 \draw [ dotted,thick, ] (8.8,2.7) -- (8.8,1.7);
  \node at (8.9,1.3) {$\alg S_{2}$};
    
 \node at (9.8,2.5) {$ {=}$};

   \fill (10.3,3) circle (0.03);
 \node at (10.5,3.1) {$1$}; 
 \fill  (10.3,2.7) circle (0.03);
 \node at (10.5,2.7) {$y$}; 
  \fill (10.3,1.7) circle (0.03);
  \node at (10.55,1.8) {$y^{2}$}; 
 \draw [ dotted,thick,] (10.3,2.7) -- (10.3,1.7);
 
 \fill  (10.3,1.5) circle (0.03);
 \node at (10.55,1.5) {$x$}; 
  \fill (10.3,0.5) circle (0.03);
  \node at (10.8,0.5) {$x^{2} = 0$}; 
 \draw [ dotted,thick, ] (10.3,0.5) -- (10.3,1.5);  
 
 
    \fill (5.2,-2) circle (0.03);
 \node at (5.4,-1.9) {$1$}; 
 \fill  (5.2,-2.2) circle (0.03);
 \node at (5.45,-2.2) {$z$}; 
  \fill (5.2,-3.2) circle (0.03);
  \node at (5.7,-3.2) {$z^{2} = 0$}; 
 \draw [ dotted,thick, ] (5.2,-2.2) -- (5.2,-3.2);
    \node at (5.3,-3.7) {$\alg S_{3}$};
 
  \node at (4.1,-2) {$\oplus_{\tau_{3}}$};
  
     \fill (2.2,-0.7) circle (0.03);
 \node at (2.4,-0.7) {$1$}; 
 \fill  (2.2,-1) circle (0.03);
 \node at (2.4,-1) {$y$}; 
  \fill (2.2,-2) circle (0.03);
  \node at (2.45,-1.9) {$y^{2}$}; 
 \draw [ dotted,thick] (2.2,-2) -- (2.2,-1);
 \fill  (2.2,-2.2) circle (0.03);
 \node at (2.45,-2.2) {$x$}; 
  \fill (2.2,-3.2) circle (0.03);
  \node at (2.85,-3.2) {$x^{2} = 0$}; 
 \draw [ dotted,thick, ] (2.2,-2.2) -- (2.2,-3.2);
    \node at (2.4,-3.7) {$\alg S_{1} \oplus_{\tau_{2}} \alg S_{2}$};
 
 \node at (6.5,-2) {$ {=}$};
 
   \fill (7.5,-0.7) circle (0.03);
 \node at (7.7,-0.7) {$1$}; 
 \fill  (7.5,-1) circle (0.03);
 \node at (7.7,-1) {$z$}; 
  \fill (7.5,-2) circle (0.03);
  \node at (7.75,-1.9) {$z^{2}$}; 
 \draw [ dotted,thick] (7.5,-2) -- (7.5,-1);
 \fill  (7.5,-2.2) circle (0.03);
 \node at (7.75,-2.2) {$y$}; 
  \fill (7.5,-3.2) circle (0.03);
  \node at (7.75,-3.1) {$y^{2}$}; 
 \draw [ dotted,thick, ] (7.5,-2.2) -- (7.5,-3.2);
 \fill  (7.5,-3.4) circle (0.03);
 \node at (7.75,-3.4) {$x$}; 
  \fill (7.5,-4.4) circle (0.03);
  \node at (8,-4.5) {$x^{2} = 0$}; 
 \draw [ dotted,thick, ] (7.5,-3.4) -- (7.5,-4.4);
}
 \end{tikzpicture}\end{center} \caption{The iterated partial gluing of three simple 2-potent chains.} \label{fig:simplegluings}
 \end{figure}
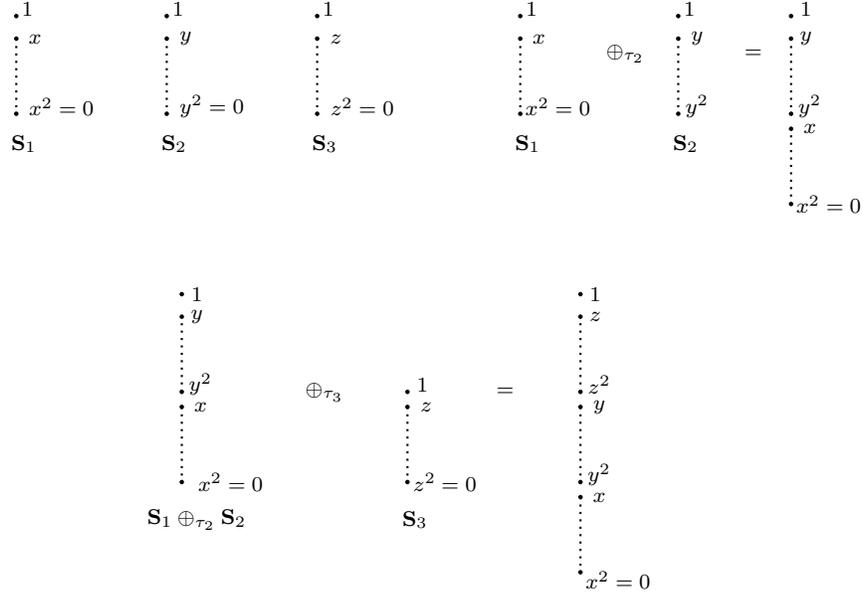
  Consider the triple $(\alg S_{1}, \sigma_{1}, \gamma_{1})$ where the implication in $\alg S_1$ is redefined to be: $x \to y = 1$ iff $x \leq y$, and undefined otherwise, and furthermore for all $a$ with $x^{2}\leq a \leq x$, we have $$\sigma_1(a) = y^{2}, \quad \gamma_1(a) = y,\quad \sigma_1(1) = \gamma_1(1) = 1.$$
Moreover, consider the triple $(\alg S_{2}, \ell_{2}, r_{2})$ where the maps $\ell_{2}, r_{2}$ have empty domain.
\begin{lemma}\label{lemma:iter}
$(\alg S_{1}, \sigma_{1}, \gamma_{1})$ is a lower-compatible triple and $(\alg S_{2}, \ell_{2}, r_{2})$ is an upper-compatible triple. 
\end{lemma}
\begin{proof}
For all $a, b \in S_{1}-\{1\}$, we have $\sigma(a) \leq b\leq \gamma(a)$, which implies that the two operators form a residuated pair. In a lower-compatible triple the implication $x \to y$ is defined iff $\sigma(x) \leq y$ and $x \not\leq y$, and we can show that this holds in $(\alg S_1, \sigma_1, \gamma_1)$.     
Indeed, notice that for all $x, y$ such that $y^2 \leq x$, we have  $\sigma_1(x) \leq y$ and by definition $x \to y$ is undefined if and only if $x \not\leq y$. Also, it follows from direct computation that $\sigma_{1}$ is a strong conucleus, $\gamma_{1}$ is a closure operator, and $cd=dc \leq \sigma_{1}(c)$ for all $c, d \in S_{1}, d \neq 1$. Thus $(\alg S_{1}, \sigma_{1}, \gamma_{1})$ is a lower-compatible triple. 

$(\alg S_{2}, \ell_{2}, r_{2})$ is an upper-compatible triple since all the products are defined and all other properties are vacuously true.
\end{proof}
Moreover, we can consider the ideal $I = \{0\}$ with $\top_{I} = 0$, where $0$ is the bottom element of the chain, and both assumptions $A1, A2$ are satisfied. Also, conditions $A3,A4$ are trivially satisfied since the algebras considered are chains. Thus, letting $\tau_{2} = (\sigma_{1}, \gamma_{1}, \ell_{2}, r_{2})$, we can define the partial gluing $\alg S_{1} \oplus_{\tau_{2}} \alg S_{2}$, that is a total IRL since $\alg S_{2}$ has a coatom (see Theorem \ref{prop:partialgluing}). 

Similarly, we consider the upper-compatible triple $(\alg S_{3}, \ell_{3}, r_{3})$ where again $\ell, r$ have empty domain. Now, we also consider 
$\alg S_{1} \oplus_{\tau_{2}} \alg S_{2}$
 where $x \to y$ is defined and equal to $1$ iff $x \leq y$, and where for all $a: x^{2}\leq a \leq x$, and $b: y^{2}\leq b \leq y$ we have $$\sigma(a) = x^{2}, \quad \gamma(a) = x,\quad \sigma(b) = y^{2}, \quad \gamma(b) = y,\quad \sigma(1) = \gamma(1) = 1.$$
Building on the same line of reasoning as Lemma \ref{lemma:iter}, $(\alg S_{1} \oplus_{\tau_{2}} \alg S_{2}, \sigma, \gamma)$ is a lower compatible triple, and we can define the partial gluing $(\alg S_{1} \oplus_{\tau_{2}} \alg S_{2}) \oplus_{\tau_{3}} \alg S_{3}$ where $\tau_{3} = (\sigma, \gamma, \ell, r)$. This process can be iterated, and provides a way of constructing a partial gluing of a finite family of simple chains, indexed by a totally ordered set of indexes. Let us now give a more general definition, in order to be able to construct a partial gluing of a family of algebras, indexed by an arbitrary totally ordered chain with a largest element. 

We consider a family of algebras $\{\alg A_{i}\}_{i \in \II}$, where: each $\alg A_i$ is a simple 2-potent CIRL-chain with a coatom $c_i$ and a bottom $0_i$; $A_i \cap A_j = \{1\}$ for all $i, j \in I$; $\II = (I, \leq)$ is a totally ordered index set with largest element $i_{0}$. We will now define the \emph{iterated partial gluing of $\{\alg A_{i}\}_{i \in I}$}, and denote it with $\bigoplus_{I}\alg A_{i}$, as follows. The domain of $\bigoplus_{I}\alg A_{i}$ is given by $\bigcup_{i \in I} A_i$. The order is defined by $x \leq y $ iff either:
\begin{enumerate}
	\item $x, y \in A_i$ for some $i \in I$ and $x \leq_{A_i} y$; or
	\item $x \in A_i, y \in A_j$ and $i < j$,
\end{enumerate}
For each $\alg A_i$, and for all $x \neq 1$, let $\sigma_i(x) = 0_i$, $\gamma_i(x) = c_i$, and $\sigma_i(1) = \gamma_i(1) = 1$. 
The product and implication are as follows:
\begin{align*}
x \cdot y &=\left\{\begin{array}{ll}
x \cdot_{A_i} y & \mbox{ if } x, y \in  A_i \mbox{ for some } i \in I\\
\sigma_i(x) & \mbox{ if } x \in A_i, y \in A_j \mbox{ and } i < j\\
\sigma_j(y) & \mbox{ if } x \in A_i, y \in A_j \mbox{ and } j < i
\end{array}\right.\\
x\to\, y &=\left\{\begin{array}{ll} 
c_{i_0} & \mbox{ if } x, y \in  A_i \mbox{ for some } i \in I \mbox{ and } x \not\leq y\\
\gamma_j(y) & \mbox{ if } x \in A_i, y \in A_j \mbox{ and } j < i\\
1 & \mbox{ if } x \leq y \end{array}\right.\\
\end{align*}
\begin{proposition}
	Let $\{\alg A_{i}\}_{i \in I}$ be a family of simple 2-potent CIRL-chains $\alg A_i$, each with a coatom $c_i$ and a bottom $0_i$, such that $A_i \cap A_j = \{1\}$ for all $i, j \in I$, and $\II = (I, \leq)$ a totally ordered index set with a largest element $i_{0}$. Then the iterated partial gluing $\bigoplus_{I}\alg A_{i}$ is a CIRL.
\end{proposition}
Moreover, if $I$ is a totally ordered finite set of indexes, the above definition of iterated partial gluing corresponds to iterating the partial gluing construction as in the above example with the algebras $\alg S_1, \alg S_2, \alg S_3$. Indeed:
\begin{lemma}
Let $\{\alg A_{i}\}_{i \in I}$ be a family of simple 2-potent CIRL-chains $\alg A_i$, each with a coatom $c_i$ and a bottom $0_i$, such that $A_i \cap A_j = \{1\}$ for all $i, j \in I$.
	Let $I$ be a totally ordered set with largest element $i_0$ such that $I - \{i_0\}$ has a largest element $i_1$. Then: 
	\begin{enumerate}
		\item\label{lemma:iterated1} $(\bigoplus_{I-\{i_0\}}\alg A_{i}, \sigma, \gamma)$ is a lower compatible triple, where: the implication is redefined to be $x \to y = 1$ iff $x \leq y$, and undefined otherwise; if $x \in A_i -\{1\}$, $\sigma(x) = \sigma_i(x) = 0_i, \gamma(x) = \gamma_i(x) = c_i$, $\sigma(1) = \gamma(1) = 1$.
		\item\label{lemma:iterated2} $(\alg A_{i_0}, \ell, r)$ is an upper-compatible triple where $\ell$, $r$ have empty domain. 
		\item\label{lemma:iterated3} $\bigoplus_{I}\alg A_{i} \cong  (\bigoplus_{I-\{i_0\}}\alg A_{i})\oplus_{\tau}\alg A_{i_0}$, with $\tau = (\sigma, \gamma, \ell, r)$ and $I = \{0\}$ with $\top_{I} = 0$.
	\end{enumerate}
\end{lemma}

We will now show how to characterize $\mathsf{GL}_{2}$-chains with iterated partial gluings.
 For a $\mathsf{GL}_{2}$-chain $\alg A$ and $a \in A$, let us now define $A(a)=\{x \in A: x^2=a^2\}$.

\begin{lemma}\label{lemma:iteratedchains}
 Let $\alg A$ be a $\mathsf{GL}_{2}$-chain and $a,b \in A$. Then:
	\begin{enumerate}
		\item 	If $a \leq b <1$, then $ab=\min A(a)$. 
		\item If $a\leq b$, then $a \to b=1$. 
		\item If $a<b<1$ and $A(a) =A(b)$, then $\alg A$ has a coatom $c$ and $b \to a=c$.
		\item If $a<b<1$ and $A(a)\not =A(b)$, then $b \to a=\max  A(a)$. 
	\end{enumerate}
Moreover, if $\alg A$ has no coatom, then it is a G\" odel algebra.
\end{lemma}
\begin{proof}

For (1), note that if $a\leq b < 1$, then using the defining property for $\mathsf{GL}_{2}$, we get $a^2 \leq ab \leq (a \land b)^{2} = a^{2}$ so $ab=a^2 = \min A(a)$. (2) always holds in CIRLs.

Let us prove (3). If $a<b < 1$ and $A(a) = A(b)$, then for every non-identity element $c$ of $A$, we have $bc \leq b^2 =a^2 \leq a$, and so $c \leq b \to a$; therefore $\alg A$ has a coatom, which is equal to $b \to a$, for all  $a<b\not = 1$ with $A(a) = A(b)$. If $A(c)$ is a singleton for all $c \in A$ (i.e. $\alg A$ is a G\" odel algebra), then $\alg A$ may have no coatom, but if there is at least one non-trivial $A(c)$, then $\alg A$ has a coatom.

For (4) suppose $a<b <1$ and $A(a) \not = A(b)$, then for every $c \in A(a)$, we have $bc =a^2 \leq a$, but for $d>a$ with $A(d) \not = A(a)$, we have $bd \geq (b \wedge d)^2 \not = a^2$, so $bd \not \in A(a)$, hence $bd \not \leq a$. Therefore, $A(a)$ has a maximum element and $b \to a= \max A(a)$, for all non-identity $b>a$ with $A(a) \not = A(b)$.  

If $\alg A$ has no coatom, then $A(a) = \{a\}$ for each $a \in A$ by (3), thus every element of $\alg A$ is idempotent. Therefore, $\alg A$ is a G\"odel algebra.
\end{proof}
	We are now ready to characterize $\mathsf{GL}_{2}$-chains as iterated partial gluings of simple chains. 
	
	\begin{proposition}
		The chains in $\mathsf{GL}_{2}$ are exactly the iterated partial gluings of simple bounded CIRL-chains with a coatom over a totally ordered index set with both a bottom and a top element.
	\end{proposition}
	\begin{proof}
		For a chain $\alg A$ in  $\mathsf{GL}_{2}$ we denote by $A^2=\{a^2 : a \in A\}=\{a : a=a^2\}$  the set of idempotent elements of $A$ (equivalently all squares of $\alg A$). 
Note that for $a, b \in A$ we have $A(a)=A(b)$ iff $a^2=b^2$ ; also for every $a \in A$ we have $A(a)=A(a^2)$. Moreover, if $A(a) \not = A(b)$, then $a<b$ iff for all $x \in A(a)$ and $y \in A(b)$ we have $x<y$. Therefore, the collection $\{A(a):a \in A\}$ is equal to the collection $\{A(s):s \in A^2\}$, it partitions $A$ into equivalence classes, which are intervals, and these intervals are linearly ordered in $\alg A$; also $A(1)=\{1\}$.  So, $A$ is the order-theoretic ordinal sum of the chains $A(s)$ along $A^2$.

Moreover, we have seen in the proof of Lemma \ref{lemma:iteratedchains} that for all $a \in A$, $A(a)$ has a maximum element whenever there is a $b<1$ strictly larger that $a$. 
Notice now that if such a $b$ does not exist, then $A(a)$ is the biggest interval below $1$, which necessarily has a maximum element by the preceding paragraph (the coatom of $\alg A$), unless $\alg A$ is a  G\" odel algebra.  But even if $\alg A$ is a  G\" odel algebra, then $A(a)=\{a\}$, so $A(a)$ has a maximum element. 

We define $I:=\{\max A(a): a \in A-\{1\}\}$, the set of all these maximal elements, and note that $\{A(i): i \in I\}=\{A(a): a \in A-\{1\}\}$. 
For $i \in I$ we define the set $A_i:=A(i)\cup\{1\}$ and note that it supports the structure of a simple $2$-potent integral residuated chain $\alg A_i$; the order and the multiplication are inherited by $\alg A$ and all interesting divisions produce the coatom of $\alg A_i$. We mention that the structure of each $A(i)$, for $i\in I$, is that of an arbitrary bounded chain. 

We now claim that $\alg A$ is the iterated gluing of the algebras $\{\alg A_{i}\}_{i \in \II}$, where for $x \in A(i)$, $\sigma(x)=\min A(i)$ and $\gamma_i(x)=\max A(i)(=i)$. 
 Indeed it follows from the definition that the domain and the order coincide. The monoidal operation inside each $\alg A_{i}$ coincides with the one inherited from in $\alg A$: for $z, w \in A(i)$ we have $zw = z^{2}=w^2$. Also, for $x \in A_{i}, y \in A_{j}$ with $j < i$, we have $xy= y^{2} = \sigma(y)$. For the implications $x \to y$ with $y <x$, we have that if $A(x) =A(y)$, then $x \to y =c$, the coatom of the chain $\alg A$, and if  $A(y)=A(i) \neq A(x)$ where $i=\max A(y) \in I$, then $x \to y =\max A(y)=\gamma_i(y)$.

We now show that given any family $\{\alg A_{i}\}_{i \in I}$ of simple 2-potent chains (with a coatom), their iterated partial gluing belongs to $\mathsf{GL}_{2}$. Indeed if $x, y\in A_{i} - \{1\}$, then $x \cdot (x \land y) = x^{2} = y^{2} = (x \land y)^{2}$. Also, if $x \in A_{i}-\{1\}, y \in A_{j}-\{1\}, j < i$, then $x \cdot (x \land y) = x \cdot y = \sigma(y) = y^{2} = (x \land y)^{2}$, and $x \cdot (x \land y) =x \cdot x = (x \land y)^{2}$. As residuated lattices are determined by their order and multiplication reducts, the result follows. 
	\end{proof}
	
	We are now going to show that $\mathsf{GL}_{2}$ is generated by its finite members, that is, it has the finite model property (or FMP). First we need the following technical lemma. 
\begin{lemma}\label{lemma:generation}
Let $\alg A$ be a chain in $\mathsf{GL}_{2}$,  $X$ a subset of $A$ and $\langle X \rangle$ the subalgebra of $\alg A$ generated by $X$.
\begin{enumerate}
\item If $X$ consists solely of idempotents, and for all $x \in X$, $A(x)$ is a singleton except possibly for $A(m)$ in case $X-\{1\}$ has a maximum element $m$, then 
$\langle X \rangle=\{1\} \cup X $.
\item Otherwise, if either $X$ contains a non-idempotent element or if $A(x)$ is not a singleton for some non-maximal element $x$ of $X-\{1\}$, then $\alg A$ has a coatom $c$ and 
$$\langle X \rangle=X \cup \{\min A(x): x \in X\} \cup \{\max A(x): x \in X\} \cup \{1, c, \min A(c)\}.$$ 
\end{enumerate}
 \end{lemma}

\begin{proof}
Clearly, $1$ needs to be in $\langle X \rangle$. Also, closure under the lattice operations does not increase the original set. From Lemma~\ref{lemma:iteratedchains}, $x^2=\min A(x)$. Also, the only other elements that are generated are: (i) the coatom $c$, when there is an  $A(x)$ such that  $A(x) \cap X$  is not a singleton (i.e., $X$ does not consist solely of idempotents) and (ii) $\max A(x)$, when there is a block $A(y) \cap X$ strictly between $A(x) \cap X$ and $A(1)$, as well as the ones obtained by interactions of (i) and (ii).

1. In the first case, $X$ is closed under multiplication, as for $x,y \in X-\{1\}$, we have $xy=(x \wedge y)^2 \in \{x^2, y^2\}=\{x,y\}$. Also, for $x,y \in X-\{1\}$ with $x<y$, the set $A(x)$ is a singleton $\{x\}$ and so $A(x)\not = A(y)$, hence $y \to x=\max A(x)=x$. 

2. In the second case, either $X$ has a non-idempotent element, which also is present in $\langle X \rangle$, or every element in $X$ is idempotent and there is some non-maximal element $x$ of $X-\{1\}$ such $A(x)$ is not a singleton, in which case there also exists a $y>x$ in $X-\{1\}$ (as $x$ is non-maximal there) and $A(y)\not = A(x)$ as $x,y$ are distinct idempotents; hence $y \to x = \max A(x)$, where $\max A(x) \not=x$, as $A(x)$ is not a singleton, hence $z$ is not idempotent in  $\langle X \rangle$. In any case,  $\langle X \rangle$ contains a non-idempotent element $w$. Then $\alg A$ has a coatom $c$, which is equal to $w \to w^2$ and which is in $\langle X \rangle$.
Closure under multiplication is equivalent to closure under squares, which is equivalent to containing the bottom of the block of an element (since $a^2=\min A(a)$, for all $a \in A$.) We need to consider only implications of the 
 form $a \to b$, for $b<a\not =1$. If $A(a)=A(b)$, then $a \to b=c$, and if  $A(a)\not =A(b)$ then $a \to b= \max A(b)$. Conversely, if $b \in \langle X \rangle-\{1,c\}$, $\max A(b)=c \to b \in \langle X \rangle$, since $c \in \langle X \rangle$; the special case of $b=c$ also works as $c=\max A(c)$.
\end{proof}
 We are now ready to show the following.

 \begin{proposition}
The variety $\mathsf{GL}_{2}$ is locally finite, hence it has the FMP. 
\end{proposition}
\begin{proof}
By Lemma~\ref{lemma:generation}, if $X$ is finite of size $n$, then $\langle X \rangle$ is also finite of size at most $3n+3$.
\end{proof}
We now use the previous results to show that the AP fails in the class of chains in $\mathsf{GL}_{2}$. 
\begin{theorem}\label{lemma:amfail}
 The amalgamation property fails for the class of $\mathsf{GL}_{2}$-chains.
\end{theorem}

\begin{proof}
Let $\alg A$ be the $3$-element G\" odel algebra where $A=\{0<a<1\}$, and also let  $\alg B$ and $\alg C$ be the $\mathsf{GL}_{2}$-chains specified by the following ordered blockings $B=\{\{0\}< \{a<b\} <\{1\}\}$ and $C=\{\{0\}< \{a\} <\{c\} <\{1\}\}$; so  $b^2=a$, and the other elements are idempotent. Note that $\alg A$ is a common subalgebra of $\alg B$ and $\alg C$, even though $A$ does not contain the top of $B(a)=\{a<b\}$, which is $b$. Let $\alg D$ be a $\mathsf{GL}_{2}$-chain that is an amalgam of this $V$-formation.  Since $a<c$ in $\alg C$, the same is true in $\alg D$. Since $\alg D$ is a $\mathsf{GL}_{2}$-chain, we have that $c \to a$ is the top of $D(a)$. However, since $b^2=a$, we have $b \in D(a)$, hence $b \leq c \to a$ in $\alg D$. Since $c \to a = a$ in $\alg C$, this yields $b \leq a$, a contradiction. 
\end{proof}
Interestingly, we can characterize exactly when the AP fails for $\mathsf{GL}_{2}$-chains. 
\begin{proposition}
 A $V$-formation $\alg A$, $\alg B$, $\alg C$ of $\mathsf{GL}_{2}$-chains (WLOG we assume that $A \subseteq B, C$) fails to have a $\mathsf{GL}_{2}$-chain amalgam iff
$\alg A$ is a G\" odel algebra and there is $a \in A$ such that $B(a)$ is not singleton, $C(a)$ is a singleton, and $C(a)$ is not the maximum nontrivial block of $\alg C$ (or the same with $\alg B$ and $\alg C$ swapped). 
\end{proposition}

\begin{proof}
The right-to-left direction follows from the same argument in the proof of Theorem \ref{lemma:amfail}. 

We prove the other direction.
 In case $\alg A$ has some non-idempotent element, then it has a coatom and hence so do $\alg B$ and $\alg C$, and the coatom of $\alg A$ coincides with his copy in $\alg B$ and also  in $\alg C$.  Also, if $\{\alg A_{i}\}_{i \in I}$, $\{\alg B_{j}\}_{j \in J}$, and $\{\alg C_{k}\}_{k \in K}$ are the associated decompositions, then $I \subseteq J, K$ and for $a \in A$, the chains $A(a), B(a), C(a)$ share the same top and bottom. We may assume that $I = J \cap K$, and we take $J \cup K$ as the index set for $\alg D$; the order on  $J \cup K$ is any amalgam of the chain $V$-formation given by $I, J, K$. Then for $i \in I$, we take $D(i)$ to be any amalgam of the bounded chain $V$-formation given by $A(i), B(i), C(i)$; for $j \in J-I$ we take $D(j)=B(j)$ and for $k \in K-I$ we take $D(k)=C(k)$.

Now assume that $\alg A$ is a G\" odel algebra. For each $a \in A$, we define $D(a)=B(a) \cup C(a)$ and the rest of $\alg D$ is defined as above, except for one case. If both $\alg B$ and $\alg C$ have coatoms $c_B, c_C$ not in $\alg A$, $B(c)$ and $C(c)$ are merged in the obvious way. It is easy to see that $\alg B$ and $\alg C$ are subalgebras of $\alg D$ with respect to multiplication. Implication could create a problem by producing the top of $B(a)$ and of $C(a)$, as they could be different elements. This can happen only when one of them is a singleton, say $C(a)$, and the other is not, say $B(a)$. Also, this can happen only of the implication is of the form $d \to x$, where $x \in D(a)$, $x<d<1$, $d \not \in C$ and $D(x) \not = D(d)$. But this is impossible by the assumption.
\end{proof}
Using the work in \cite{FussMet}, we can actually prove that the AP fails for the variety $\mathsf{GL}_2$. According to the authors a class of algebras $\vv K$ is said to have the \emph{one-sided amalgamation} property, or the 1AP, if  every  V-formation $(\alg A, \alg B, \alg C, i: A \to B, j: A \to C)$ in $\vv K$ has a \emph{1-amalgam} $(\alg D, h: B \to D, k: C \to D)$  in $\vv K$, i.e., $\alg D \in \vv K$, $k$ is an embedding, $h$ is a homomorphism, and $h \circ i = k \circ j$. Notice that if $h$ is an embedding we have the usual notion of amalgam. 
Given a variety $\vv V$, let $\vv V_{FSI}$ the class of finitely subdirectly irreducible members of $\vv V$. In \cite[Theorem 3.4]{FussMet}, the authors show that if a variety $\vv V$ has the congruence extension property and  $\vv V_{FSI}$ is closed under subalgebras, then $\vv V$ has the AP if and only if $\vv V_{FSI}$ has the 1AP. Since $\mathsf{GL}_2$ is congruence distributive, and the finitely subdirectly irreducibles are exactly the nontrivial chains (which are clearly closed under subalgebras), we can apply the mentioned result.
\begin{proposition}
	The amalgamation property fails for $\mathsf{GL}_2$.
\end{proposition}
\begin{proof}
Consider again the V-formation $(\alg A, \alg B, \alg C, i, j)$ in the proof of Theorem \ref{lemma:amfail}, where $i, j$ are the inclusion maps and note that all three algebras are FSI. We show that every 1-amalgam for it is an actual amalgam. Since we have shown that it has no amalgams, via \cite[Theorem 3.4]{FussMet} this concludes the proof. 

In particular, we prove that if $(\alg D, h: B \to D, k: C \to D)$ is a 1-amalgam, then $h$ is necessarily injective. Notice that the composition $h \circ i = k \circ j$ is injective, since both $j$ and $k$ are injective. Thus, no elements of $\alg A$ are collapsed by $h \circ i$. In particular, $h$ does not collapse $i(1)$ and $i(a)$, i.e., $i(a) \notin {\rm ker}(h)$. Therefore, also $i(b) \notin {\rm ker}(h)$, since otherwise $i(b)^2 = i(b^2) = i(a)$ would be in the kernel as well, and it is not. We conclude that the kernel of $h$ is trivial, and hence, $h$ is injective.
\end{proof}
We remark that if we modify $\mathsf{GL}_{2}$-algebras to be expansions of IRLs with a new constant $c$, and be such that for the finitely subdirectly irreducible members they satisfy $x=1$ or $x \leq c$, then the subalgebras will also contain $c$ and amalgamation will hold for all such $\mathsf{GL}_{2}$-chains. Then by Theorem 49 in \cite{MMT}, which allows one to extend the amalgamation property from the FSI members to the whole variety, the amalgamation property extends to the variety of all modified $\mathsf{GL}_{2}$-algebras.
\medskip 
\section*{Funding}
This work has received funding from the European Union's Horizon 2020 research and innovation programme under the Marie Sk\l odowska-Curie grant agreement No 890616 awarded to Ugolini.

\end{document}